\documentclass{amsart}
\usepackage{amsmath,graphicx}

\usepackage{latexsym}

\newtheorem{thm}{Theorem}
\newtheorem{lemma}[thm]{Lemma}
\newtheorem{cor}[thm]{Corollary}
\newtheorem{prop}[thm]{Proposition}

\theoremstyle{remark} 
\newtheorem{remark}[thm]{Remark}
\newtheorem{example}[]{Example}

\newcommand{\pd}{\partial}

\newcommand{\RR}{\mathbb{R}}
\newcommand{\RP}{\mathbb{R}\textrm{P}}
\newcommand{\EE}{\mathbb{E}}
\newcommand{\PP}{\mathbb{P}}

\newcommand{\B}{B_{x}}%barrier%

\newcommand{\e}{\varepsilon}

\newcommand{\N}{\mathcal{N}}

\newcommand{\const}{c}
\newcommand{\Const}{C}

\title[random algebraic hypersurfaces]{On the number of connected components of random algebraic hypersurfaces}
\author{Yan V. Fyodorov}
\author{Antonio Lerario}\author{Erik Lundberg}

\begin{document}

\begin{abstract}
We study the expectation of the number of components $b_0(X)$ of a random algebraic hypersurface $X$
defined by the zero set in projective space $\RP^n$ of a random homogeneous polynomial $f$ of degree $d$.
Specifically, we consider \emph{invariant ensembles}, that is Gaussian ensembles of polynomials that are invariant under an orthogonal change of variables.

Fixing $n$, under some rescaling assumptions on the family of ensembles (as $d \rightarrow \infty$), 
we prove that $\EE b_0(X)$ has the same order of growth as $\left[ \EE b_0(X\cap \RP^1) \right]^{n}$.
This relates the average number of components of $X$ to the 
classical problem of M. Kac (1943) on the number of zeros of the random univariate polynomial $f|_{\RP^1}.$

The proof requires an upper bound for $\EE b_0(X)$, which we obtain by counting extrema
using Random Matrix Theory methods from \cite{Fyodorov}, 
and it also requires a lower bound, which we obtain by a modification of the barrier method from \cite{LerarioLundberg, NazarovSodin}.

%To be more precise about the rescaling assumption on the eigenspace weights, 
%let $\{Y_{\ell}^{j}\}_{j\in J_\ell}$ denote the standard basis of spherical harmonics (eigenfunctions of the spherical Laplacian) 
%of degree $\ell$ on $S^n$, 
%a random invariant polynomial $f(x)$ of degree $d$ in $n+1$ variables can be constructed as a sum of weighted spherical harmonics with i.i.d. Gaussian coefficients $\xi_{\ell}^j$:
% \begin{equation*}
% f(x)=\sum_{d-\ell\in 2\mathbb{N}}{p_d(\ell)}\sum_{j\in J_\ell}\xi_{\ell}^j \|x\|^{d-\ell}Y_{\ell}^j\left(\frac{x}{\|x\|}\right),\quad {p_d(\ell)}\geq0.
% \end{equation*}
%The nonnegative weights $p_d(d), p_d({d-2}), \ldots$ parameterize all invariant ensembles. 
%We assume that, after the renormalization $\sum_\ell{p_d(\ell)}=1$, there exists $0<\lambda\leq1$ and a nonzero function $\psi$ such that as $d \rightarrow \infty$:
%$$p_d({d^{\lambda} x})d^{\lambda} \rightarrow \psi(x) \quad \textrm{pointwise and dominated by a subgaussian function}.$$
We also provide  \emph{quantitative} upper bounds on implied constants; for the \emph{real Fubini-Study}  model these estimates reveal 
super-exponential decay (as $n \rightarrow \infty$) of the leading coefficient (in $d$) of $\EE b_0(X)$.
%$$\limsup_{n\to \infty} \limsup_{d\to \infty} \left(\frac{\EE b_0(X)}{d^{\lambda n}} \sqrt{\frac{\mu_{n+1}}{\mu_{n+3}}} \right)^{1/n}  \leq \frac{2}{e},$$
%where $\lambda$ and $\mu_k$ relate to the rescaling limit of eigenspace weights; 
%$d^\lambda$ is the rescaling factor, and $\mu_k$ denotes the $k$-th moment of square of the rescaling limit.
\end{abstract}

\maketitle

%\begin{center} \Small
%\end{center}

\section{Introduction}
%\subsection{A random approach to Hilbert's Sixteenth Problem}
Complex algebraic geometry is a subject of theorems that build crucially on generic statements,
whereas real algebraic geometry is by necessity a more algorithmic subject.
For instance, it is well known that a \emph{generic} complex plane algebraic curve of degree $d$ (even if the coefficients of the defining polynomial are real) 
is a surface of genus \begin{small}$g=\frac{(d-1)(d-2)}{2}$\end{small}. 
On the other hand, the real  part of this curve will consist of \emph{at most} $g+1$ components
homeomorphic to circles.
Moreover, all possibilities between \begin{small}$\frac{1+(-1)^{d+1}}{2}$\end{small} and $g+1$ can occur,
and even case-by-case analysis is difficult; in fact the number of possible arrangements increases super-exponentially when $d\to \infty$ \cite{KhOr} (the interested reader can see \cite{Wilson} for a classical exposition).

Yet, one still would like to have a broad picture in the real setting,
and a way to achieve that is to replace the word ``generic'' with ``typical''.
If we choose our curve \emph{randomly} then what outcome do we expect to see?
This question is part of a random approach 
(recently initiated by P. Sarnak \cite{Sarnak}) to \emph{Hilbert's Sixteenth Problem} (H16):
to investigate the ``number, shape, and position'' 
of the components of real algebraic hypersurfaces in projective space.

%\cite{GaWe1, GaWe2, GaWe3, Lerario2012, LerarioLundberg, LeLu2, Sarnak, Sarnak2, Sodin}.

Returning to the basic question:
$$\emph{``How many components does a random hypersurface $X\subset \RP^n$ have on average?''}$$
we need to make the meaning of ``random hypersurface'' precise.

By a \emph{hypersurface} $X$ we will always mean the zero set in the projective space $\RP^n$ of a real homogeneous polynomial of degree $d$ in $n+1$ variables - we will denote the space of such polynomials by 
$$W_{n,d}=\{\textrm{real homogeneous polynomials of degree $d$ in $n+1$ variables}\}.$$

Still we need to specify a probability distribution on $W_{n,d}$, making precise the meaning of ``random'',
which strongly affects the outcome.
Consider for example a real curve whose defining polynomial is sampled uniformly from the 
unit sphere in the $L^2_{S^n}$-norm. 
The corresponding random curve is called 
\emph{real Fubini-Study} and it has order $d^n$ many components \cite{LerarioLundberg} on average. 
On the other hand if we sample the defining polynomial uniformly 
from the unit sphere in the $L^2_{S^{2n-1}}$-norm 
(after extending the polynomial to the complex domain), 
the corresponding random curve is called \emph{Kostlan}, 
and it has  only order $d^{n/2}$ many components \cite{GaWe3}.

One goal of the current paper is to unify and ``interpolate'' these two results by extending them
to a larger family of Gaussian ensembles.
We will consider Gaussian ensembles that are invariant under an orthogonal change of coordinates 
(so that there is no preferred point or direction in projective space).

Denoting by $\{Y_{\ell}^{j}\}_{j\in J_\ell}$ the standard basis of spherical harmonics of degree $\ell$ on $S^n$, 
a random invariant polynomial $f(x)$ of degree $d$ in $n+1$ variables can be constructed as a sum of weighted spherical harmonics with i.i.d. Gaussian coefficients $\xi_{\ell}^j$:
 
\begin{equation}\label{random}
 f(x)=\sum_{d-\ell\in 2\mathbb{N}}{p_d(\ell)}\sum_{j\in J_\ell}\xi_{\ell}^j \|x\|^{d-\ell}Y_{\ell}^j\left(\frac{x}{\|x\|}\right),\quad {p_d(\ell)}\geq0.
\end{equation}
The nonnegative weights $p_d(d), p_{d}(d-2), \ldots$ parameterize all invariant ensembles \cite{EdelmanKostlan95}. 

The question that we want to address is whether it is possible to determine, from the ${p_d(\ell)}$, 
the order of growth (as $d \rightarrow \infty$) of the number of connected components of the random hypersurface $\{f=0\}.$ 
Since we will be interested in the asymptotic for $n$ fixed and $d$ large, 
it is natural to ask that our ensembles of random polynomials, one for every $d$, are ``coherent'' 
one with respect to the other -- for example taking the Kostlan distribution for even degrees and the real Fubini-Study for odd degrees
leads to erratic asymptotic behavior for the number of connected components of $X$.

We are thus led to consider choices of the weights $p_d(d), p_{d}(d-2),\ldots$ satisfying the following two conditions: (normalization) $\sum_{d-\ell \in 2\mathbb{N}}p_{d}(\ell)=1$; (coherence) there exists $0<\lambda\leq 1$ and a \emph{nonzero} function $\psi$ such that:
\begin{equation*}
p_{d}(d^{\lambda} x)d^{\lambda} \rightarrow \psi(x) \quad \textrm{pointwise and dominated by a subgaussian function.}
\end{equation*}
The first condition (normalization) is simply obtained by dividing for each $d$ all the weights by their sum and does not change the asymptotic (it only rescales the random polynomial by a constant). We will call such choice of ensembles a \emph{coherent family}. For example, in the real Fubini-Study we have $\lambda=1$ and $\psi$ is the characteristic function of the unit interval; 
in the Kostlan case $\lambda=1/2$ and $\psi$ is a standard Gaussian (see Figure \ref{figweights} and Example 1 in Section \ref{sec:Examples}).
\begin{figure}
\includegraphics[width=0.6\textwidth]{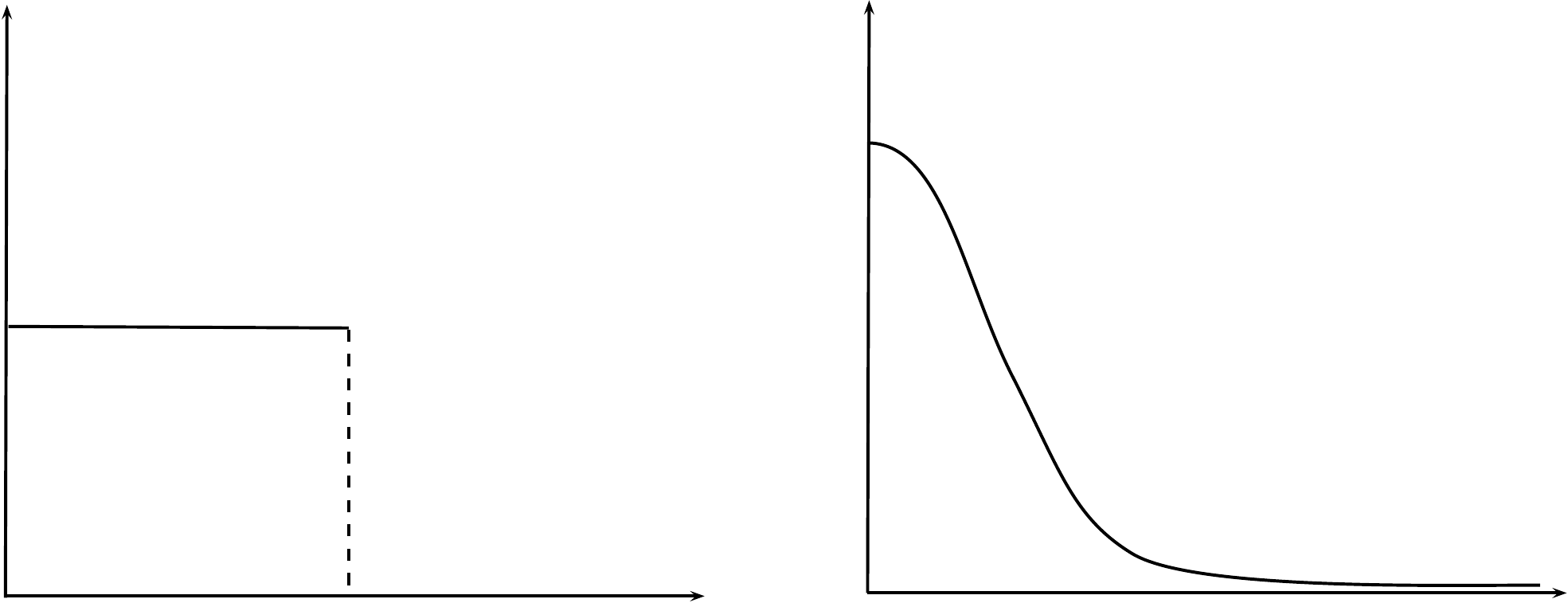}
\caption{The rescaled eigenweights for the RFS ensemble is $\psi=\chi_{[0,1]}$ (left); for the Kostlan ensemble $\psi(x)$ is a Gaussian function (right).}
\label{figweights}
\end{figure}

Given a coherent family of ensembles, 
restrict $f\in W_{n,d}$ to a fixed projective line $\RP^1$,
and consider the number of points in the set $X\cap \RP^1$.
This is the classical, and well-understood, problem of counting the number of real zeros of a univariate polynomial.
The following theorem relates the asymptotic for this number of points to the number of connected components of $X.$

\begin{thm}[Slice sampling]\label{main}
Let $X\subset \RP^n$ be a random hypersurface of degree $d$ whose defining polynomial $f\in W_{n,d}$ is sampled from a coherent family of ensembles. 
There exist positive constants $\const_n, \Const_n$ such that 
$$ \const_n \left[ \EE b_0(X\cap \RP^1) \right]^n \leq \EE b_0(X) \leq  \Const_n \left[ \EE b_0(X\cap \RP^1) \right]^n .$$
\end{thm}

In fact, using the  $\lambda>0$ in the definition of the coherent family of ensembles, we will first prove the existence of positive constants
$\hat{c}_n, \hat{C}_n$ such that
\begin{equation}\label{ord}
 \hat{c}_n d^{\lambda n} \leq \EE b_0(X) \leq  \hat{C}_n d^{\lambda n},
\end{equation}
and then we will compute $\EE b_0(X\cap \RP^1)$, proving that it has order $d^\lambda.$

Note that for the real Fubini-Study $\lambda=1$ and for the Kostlan case $\lambda=1/2$: in this way one recovers the results from \cite{GaWe3} and \cite{LerarioLundberg}.

The ideas from this paper run parallel with the more sophisticated \cite{Sodin},
where general Riemannian manifolds are considered. 
In fact, it is possible to prove, under slightly different assumptions,
that $\EE b_0(X) / d^{\lambda n}$
approaches a limit using \cite[Theorem 4]{Sodin}.
The coherence condition we imposed is replaced with a rescaling condition on the covariance structure $K_d$ of the random function. 
More precisely one defines
$$K_d(x,y)=\EE\{f(x)f(y)\}\quad x,y\in S^n,$$
and, letting $\phi:T_xS^n\to S^n$ be the exponential map, 
Nazarov and Sodin require (see Section \ref{sectionNS} for a precise formulation) that for some $0<\lambda\leq 1$:
$$K_d(\phi(d^{-\lambda}u), \phi(d^{-\lambda}v))\quad \textrm{converges uniformly on compact sets.}$$
Since in this case 
$$K_d(x,y)=\sum_{d-l\in 2\mathbb{N}}p_d(\ell)^2d(n, \ell)\tilde{C}_\ell^{\frac{n-1}{2}}(\cos \theta(x,y)),$$ 
where $\tilde{C}_\ell^{\frac{n-1}{2}}$ is a Gegenebauer polynomial, one notices some similarity between the two rescaling limits; in fact the Nazarov-Sodin rescaling condition will be satisfied in our case (see Section \ref{sectionNS} below).

The method in \cite{Sodin} uses a careful process of 
``coupling'' the original function to the auxilary function defined in $\RR^n$,
followed by a high-ground approach to studying the random function in $\RR^n$ 
using such general principles as Wiener's ergodic theorem.
Although the independent approach we take is specialized to 
invariant ensembles on the sphere,
it is more direct since we work with the original random function in the intrinsic coordinates,
and the proofs of upper and lower bounds in (\ref{ord}) are based on rather explicit constructions
that can readily yield quantitative estimates for the implied constants (see equation \eqref{constants2} below or Section \ref{sec:Examples}).
Another motivation for our more specialized setting is that
the geometrically appealing idea of ``slice sampling'' seems to lack any analogy in the general setting of Riemannian manifolds. 
We note that this aspect of our study was inspired by work on the expected volume of $X$ \cite{Kostlan} (or the Euler characteristic for $n$ odd \cite{Burgisser}) where it is shown that:
$$\EE\textrm{Vol}(X)=\textrm{Vol}(\RP^{n-1})\cdot \EE[\textrm{Vol}(X\cap \RP^1)],$$
revealing that the Kac problem on a one-dimensional slice determines the outcome (with an exact formula and explicit constant).

Moreover we believe that in this specific invariant case, where Kostlan's classification is available, 
it is natural to ask how to read the answer directly from the choice of the weights $p_d$.

For example, using our technique, we are able to prove the following upper bound for the so called Nazarov-Sodin constant \cite{Sarnak2}, i.e. the leading coefficient of $\EE b_0(X)$ as in \eqref{ord}.
In order to formulate the result, we define, for every integrable function $\psi:[0, \infty)\to \RR$ with subgaussian tail, the numbers (moments of $\psi^2$): 
$$\mu_{k}(\psi^2)=\int_{0}^{\infty}\psi(x)^2 x^{k} dx,\quad k\in \mathbb{N}.$$
The upper bound we provide is through the expected number of minima $\N_m$ on a half sphere of the random function defining $X$,
for which we provide a precise asymptotic (not just an upper bound).
This is of independent interest as $\N_m$ provides additional data about the ``landscape'' (graph of $f$).
\begin{thm}[Extrema of invariant random fields]\label{estimates}
There exist (explicit) constants $c_1, c_2, c_3>0$ 
such that for every random coherent hypersurface $X$ with rescaling exponent $\lambda$ and rescaling limit weight $\psi(x)$ we 
have:
\begin{equation}\label{constants1}
\lim_{d \to \infty} \frac{\EE b_0(X)}{d^{\lambda n}}\leq \lim_{d \to \infty } \frac{2 \EE \N_m}{d^{\lambda n}} \sim c_1\left(\frac{\mu_{n+3}(\psi^2)}{\mu_{n+1}(\psi^2)}\right)^\frac{n}{2}n^{-\frac{n}{2}-c_2}e^{-n+c_3 \sqrt{n}},
\end{equation}
where the asymptotic
is as $n \rightarrow \infty$.
\end{thm}

Here and in what follows, the notation $f(n)\sim g(n)$ means $\lim_{n\to \infty}\frac{f(n)}{g(n}=1$.

The existence of the first limit is a corollary of the fact that our rescaling assumption implies the one in \cite{Sodin} (see Section \ref{sectionNS} below); 
the asymptotic of the second limit follows directly from Theorem \ref{thm:fyodorov}.
For example, in the real Fubini-Study case one can write: 
$$\EE{b_0(X)}=a_{n}d^n+o(d^{n})\quad \textrm{(real Fubini-Study)}.$$  

Now the rescaling limit equals $\psi=\chi_{[0,1]}$, the moments of $\psi^2$ in Theorem \ref{estimates} are computable and we obtain super-exponential decay of $a_n$ (see Example \ref{RFSex} below):
\begin{equation}\label{constants2}
a_n \leq C_1n^{-\frac{n}{2}-c_2}e^{-n+c_3 \sqrt{n}}.
\end{equation}

For the case $n=2$ of curves we obtain asymptotically $\EE b_0(X)\leq \frac{d^2}{3\sqrt{3}}$ for the RFS model and $\EE b_0(X)\leq \frac{2d}{\sqrt{3}}$ for the Kostlan one (the $n$-dimensional analogue of \eqref{constants2} for the Kostlan distribution is discussed in example \ref{Kostlanex} below).

%In particular there exists a constant $C_1>0$  (bigger then $c_1$) such that:
%\begin{equation}\label{constants2}a_{n}\leq C_1n^{-c_2}e^{c_3 \sqrt{n}}\left(\frac{e}{2}\right)^{-n}\end{equation}
%(asymptotically $a_n$ decays exponentially).

We note that the statement (\ref{constants2}) cannot be deduced from any a priori bound for the number of connected components of $X$. In fact, since $\EE{b_0(X)}$ has maximal order in $d$ for the RFS model,
it is interesting to compare (\ref{constants2}) to what is known for the \emph{maximal} number of connected components of a hypersurface of degree $d$ in $\RP^n$. 
Thom-Smith's inequality (see \cite[Appendix]{Wilson}) implies that:
$$\beta_0(n,d)=\max\{ b_0(X)\,|\,\textrm{$X$ is a smooth hypersurface of degree $d$ in $\RP^n$}\}\leq d^n+O(d^{n-1}).$$
F. Bihan \cite{Bihan} has proved that indeed:
$$\lim_{d\to \infty}\frac{\beta_{0}(n,d)}{d^n}=\xi_{0,n}\quad\textrm{with}\quad  \frac{1}{2^{n-1}}\leq \xi_{0,n}\leq 1,$$
but it is not known whether the sequence $\xi_{0, n}$ converges to zero. 
Thus, the estimate \eqref{constants2} provides an interesting comparison:
$$ (a_n)^{1/n} \leq (1+o(1)) \frac{e}{\sqrt{n}}, \quad  \frac{1}{2} \leq (\xi_{0,n})^{1/n} \leq 1 .$$
A similar result is obtained for the leading order coefficient (in $d$) for the Kostlan distribution in \cite{GaWe2}
(but in that case $\EE b_0(X) \sim a'_n d^{n/2})$ does not have maximal order in $d$.

\subsection{Comments on the proof}
The statement of Theorem \ref{main}, in the form of  equation \eqref{ord}, requires to establish, for a coherent family of ensembles, the existence of two positive constants $c_n, C_n>0$ such that:
$$\hat c_nd^{\lambda n}\leq \EE b_0(X)\leq \hat C_n d^{\lambda n}.$$
The lower bound is obtained by a careful modification of the barrier method, introduced in \cite{NazarovSodin} for random harmonics on $S^2$ 
and adapted in \cite{LerarioLundberg} to random homogeneous polynomials. 
Here the barrier function will depend on the rescaled limit $\psi$ and the rescaling factor $d^\lambda$.

The upper bound uses ideas from random matrix theory
and the fact that one can bound the number of components of $X$ by the number of extrema of the random function $f|_{S^n}$
: roughly speaking, ``inside'' each component $f$ attains either a maximum or a 
minimum.
Recently the methods of finding the mean number of critical points
of any index has been considerably developed for several classes of
rotationally-invariant Gaussian random functions \cite{ABAC, ABA, Fyodorov2, Fyodorov3, Nicolaescu1, Nicolaescu2}. 
In particular, an explicit closed-form expression for the mean total number of
minima for such functions on a sphere has been derived in the lecture notes \cite{Fyodorov}
which may serve as an informal introduction to the subject. 
In our case, the coherence assumption together with the condition that our Gaussian field is polynomial 
ensure that we can compute asymptotics for the formulas presented in \cite{Fyodorov}. 
The same is true for the number of all critical points of $f|_{S^n}$, suggesting the use of standard Kac-Rice formulas, 
but this overcount yields an extra factor that grows exponentially (the average number of minima becomes a fraction of all critical points which is exponentially small, see \cite{Fyodorov}).

Concerning the number of minima itself we note the following critical point counterpart of Theorem \ref{main}.

\begin{prop}\label{prop:minima}Let $f\in W_{n,d}$ be a random polynomial sampled from a coherent family of ensembles. 
Then there exists a constant $c_n'>0$ (depending on the coherent family) such that as $d \rightarrow \infty$:
$$
\EE\#\{\textrm{\emph{minima of} $f$}\}\sim c'_n\left(\EE \#\{\textrm{\emph{minima of} $f|_{S^1}$}\}\right)^n .$$
\end{prop}

\subsection{Structure of the paper}
In Section 2, we introduce the basic notation and review Kostlan's classification of invariant ensembles. 
In Section 3 we study the upper bound for the main theorem and the lower bound is presented in Section 4. 
Section 5 is devoted to the reinterpretation of the results in terms of the reduction to the univariate case; in Section 6 we compare our results with \cite{Sodin} and in Section 7 we discuss some examples in detail.

\subsection{Notes}
Concerning the large degree asymptotic, the first breakthrough in this program
came from the work of F. Nazarov and M. Sodin in a related setting of 
random eigenfunctions, where they introduced the so-called \emph{barrier method} \cite{NazarovSodin} 
for studying the number of components of the zero set of random spherical harmonics on the two-dimensional sphere. The possibility to extend this method to homogeneous polynomials and to higher dimensions was disclosed by P. Sarnak in the letter \cite{Sarnak}.
For \emph{Kostlan} hypersurfaces of real projective manifolds this has been accomplished by D. Gayet and J-Y. Welschinger \cite{GaWe3}, and the study of random \emph{Real-Fubini-Study} hypersurfaces in $\RP^n$ was done 
in \cite{LerarioLundberg}. 
Gayet and Welschinger also began a Morse theoretic approach to the problem of bounding Betti numbers 
(i.e. not only the number of connected components) of the zero locus of random real sections of high tensor powers of a holomorphic line bundle on a real projective manifold \cite{GaWe2}.

F. Nazarov and M. Sodin have developed powerful methods for studying translation invariant Gaussian random functions in $\RR^n$ and these methods are applied  to random functions on Riemannian manifolds \cite{Sodin}; 
the sphere and the projective space are particular cases, 
and as we have mentioned the approach from \cite{Sodin} and the one followed in this paper are closely related (see Section \ref{sectionNS} below).
Similar methods were applied by P. Sarnak and I. Wigman to prove the existence of 
limit laws for topologies (including arrangements of components) 
of zero sets of band-limited functions \cite{Sarnak2}.

Focusing on a somehow different asymptotic, the study of the large number of variables was initiated in \cite{Lerario2012}, where the second author introduced a random version of a spectral sequence from \cite{Agrachev, AgrachevLerario} for the study of the topology of a random intersection of quadrics. This study was continued by the last two authors in \cite{LeLu2}, where the
last two authors proved that for a random intersection $X$ of $k$ quadrics in $\RP^n$ and for every Betti number $\lim_{n\to \infty}\EE b_i(X)=1$ (in fact even allowing dependence of $i$ on $n$ the same result still holds as long as $|i-n/2|>n^{\alpha}$,  for some $0<\alpha<1$).  A precise asymptotic for the sum of all Betti numbers was also computed in the case of an intersection of two quadrics.

\subsection*{Acknowledgements}
The authors wish to thank the Institute for Advanced Study at Princeton, where most of this research took place, and M. Sodin for stimulating comments. They also wish to thank the anonymous referee for his/her careful work and suggestions, that allowed to substantially improve the presentation.
The first author was supported by EPSRC grant EP/J002763/1 ``Insights into Disordered Landscapes via Random Matrix Theory and Statistical Mechanics''.
The second author is supported by the European Community's Seventh Framework Programme ([FP7/2007-2013] [FP7/2007-2011]) under grant agreement No. [258204].

%\subsection{Coherent families of ensembles}

\section{Random invariant polynomials}

\subsection{Some notational conventions} Let $f,g:\mathbb{N}\to [0, \infty)$. We write $f\sim g$ for $\lim_{d\to \infty}\frac{f(d)}{g(d)}=1$. Also, we use the notation $f(d)=O(g(d))$ when there exists a constant $c>0$ such that $f(d)\leq c g(d)$ for all $d\in \mathbb{N}$. Similarly $f(d)=\Theta(g(d))$ means that $f(d)=O(g(d))$ and $g(d)=O(f(d)).$

\subsection{A classification} We consider the vector space $W_{n,d}$ of real homogeneous polynomials of degree $d$ in $n+1$ variables:
$$W_{n,d}=\RR[x_0, \ldots, x_n]_{(d)}$$ 
A random \emph{invariant} polynomial is an element of the space $W_{n,d}$ endowed with a Gaussian probability distribution which is invariant by \emph{orthogonal} change of variables. 
E. Kostlan \cite{Kostlan} classified such ensembles
using the well-known 
orthogonal direct sum decomposition:
 \begin{equation}\label{decomposition}
 W_{n,d}=\bigoplus _{d-\ell \in 2\mathbb{N}}\|x\|^{d-\ell}H_{n,\ell},
 \end{equation}
into spaces $H_{n,\ell}$ of spherical harmonics of degree $\ell$ (eigenfunctions of the spherical Laplacian with eigenvalue $-\ell(\ell + n - 1)$).
Choosing an orthogonally invariant Gaussian distribution on $W_{n,d}$ amounts to specifying a scalar product on it which is invariant by orthogonal change of coordinates (then the polynomial is sampled uniformly from the unit sphere for this scalar product).

An invariant scalar product is obtained by choosing ``weights'' for each $H_{n, \ell}$. Denoting by $\{Y_{\ell}^{j}\}_{j\in J_\ell}$ the standard basis of spherical harmonics of degree $\ell$ 
 (which is an orthonormal basis with respect to the $L^2(S^n)$ norm), 
 an invariant random polynomial is exactly a polynomial of the form:
 
 \begin{equation}\label{randomf}
 f=\sum_{d-\ell\in 2\mathbb{N}}{p_d(\ell)}\sum_{j\in J_\ell}\xi_{\ell}^j \hat{Y}_{\ell}^j,\quad {p_d(\ell)}\geq0
 \end{equation}
 where the coefficients $\xi_{\ell}^j$ are independent, centered standard Gaussian (and the $p_{d}(\ell)$ are the ``weights'').  

\subsection{Covariance structure}

The random polynomial defined above is a \emph{polynomial Gaussian field}. It is well known that such a random function is determined by its covariance structure (the Gaussian assumption is fundamental):
$$G(x,y)=\EE\{f(x)f(y)\}.$$
The case when $f$ is invariant produces a covariance structure of the form:
\begin{equation}\label{covariance1}G(x,y)=\sum_{k=0}^{\lfloor \frac{d}{2}\rfloor}\beta_k\|x\|^{2k}\|y\|^{2k}\langle x, y\rangle^{d-2k}.\end{equation}
for some choice of the real numbers $\beta_0, \ldots, \beta_{\lfloor \frac{d}{2}\rfloor}$ (see \cite{Kostlan}). 

We can also write $F(x,y)$ using (normalized) Gegenbauer polynomials:
\begin{equation}\label{covariance2}G(x,y)=\|x\|^d\|y\|^d\sum_{d-\ell\in 2\mathbb{N}}r_{\frac{d-\ell}{2}}^2\tilde{C}_{\ell}^{\frac{n-1}{2}}\left(\frac{\langle x, y\rangle}{\|x\|\|y\|}\right).\end{equation}
Here $\frac{\langle x, y\rangle}{\|x\|\|y\|}$ is the cosine of the angle between $x$ and $y$ and $\tilde{C}_{\ell}^{\frac{n-1}{2}}$ is the classical Gegenbauer polynomial $P_{\ell}^{(\frac{n-1}{2})}$ (see the \cite[Appendix 7.2]{LerarioLundberg} and \cite[Section 4.3]{Kostlan}) normalized such that $\tilde{C}_{\ell}^{\frac{n-1}{2}}(1)=1$. 

In the form \eqref{covariance2} the covariance matrix is nondegenerate; the relation between the $r_i$ and the $\beta_k$ is given by \cite[eq. (7) and (8)]{Kostlan}.

Using the description of a random polynomial given in \eqref{randomf}, we can rewrite the covariance structure as:
% \begin{equation}\label{randomf}
 %f(x)=\sum_{d-\ell\in 2\mathbb{N}}{p_d(\ell)} \|x\|^{d-\ell}\sum_{j\in J_\ell}\xi_{\ell}^j Y_{\ell}^j(x),
 %\end{equation}
%we obtain:
\begin{align*}G(x,y)&=\EE\left\{\left(\sum_{d-\ell\in 2\mathbb{N}}{p_d(\ell)} \|x\|^{d-\ell}\sum_{j\in J_\ell}\xi_{\ell}^j Y_{\ell}^j\left(\frac{x}{\|x\|}\right)\right)\left(\sum_{d-\ell\in 2\mathbb{N}}{p_d(\ell)} \|y\|^{d}\sum_{j\in J_\ell}\xi_{\ell}^j Y_{\ell}^j\left(\frac{y}{\|y\|}\right)\right)\right\}\\
&=\sum_{d-\ell\in 2\mathbb{N}}{p_d(\ell)}^2 \|x\|^{d}\|y\|^{d}\sum_{j\in J_\ell} Y_{\ell}^j\left(\frac{x}{\|x\|}\right)Y_{\ell}^j\left(\frac{y}{\|y\|}\right)\\
&= \|x\|^{d}\|y\|^d\sum_{d-\ell\in 2\mathbb{N}}{p_d(\ell)}^2 \frac{d(n,\ell)}{|S^{n}|}\tilde{C}_{\ell}^{\frac{n-1}{2}}\left(\frac{\langle x, y\rangle}{\|x\|\|y\|}\right).
\end{align*}
In the second line we have used the independence of the $\xi_{l}^j$ and the fact that they are standard normals; in the third line we have used the \emph{addition formula} \cite{SW} for spherical harmonics and we have denoted by $d(n,\ell)$ the dimension of $H_{n,\ell}$:
\begin{equation}\label{eq:dimharmonic}d(n,\ell)=\frac{(n+2\ell-1)(n+\ell-2)!}{\ell! (n-1)!}=\Theta(\ell^{n-1}).\end{equation}
As a consequence we obtain the relation between our weights ${p_d(\ell)}$ and the $r_i$ given in \cite{Kostlan}:
\begin{equation}\label{eq:ptor}p_d({\ell})=r_{\frac{d-\ell}{2}}\sqrt{\frac{|S^n|}{d(n,\ell)}}.\end{equation}

\subsection{Coherent ensembles of random polynomials}
In the sequel we will be interested in the average number of components $b_0(X)$ on the sphere 
(or the projective space) of the zero set $X$ of a random invariant polynomial of degree $d$. 
In order to formulate a more precise question, 
we assume some ``coherence'' on the behavior of the weights as $d\to \infty.$

We start by noticing that our statistic, $b_0$, is invariant by dilation (i.e. the zero set of $f$ and of a nonzero multiple of $f$ have the same number of components), thus renormalizing the weights doesn't change our asymptotic (it only affects the simplicity of presentation), 
and we will assume that:
\begin{equation}\label{norm}\sum_{d-\ell\in 2\mathbb{N}}{p_d(\ell)}=1.\end{equation}
The main assumption we make on the coherence of the family of ensembles of random polynomials as $d$ goes to infinity concerns a rescaling limit of our weights. 
More precisely we consider for every $0<\lambda\leq 1$ the sequence of functions $\{P_{d,\lambda}:[0, \infty)\to \RR\}_{d\geq 0}$ defined by:
$$P_{d,\lambda}(x)=p_d(d^\lambda x)d^\lambda.$$
Here, we assume that we have extended $p_d$ to all of $[0,\infty)$ as a step function supported on $[0,d]$;
namely, over an interval $(k,k+1)$ in this range, the value of $p_d$ is chosen constant 
and equal to either $p_d(k)$ or $p_d(k+1)$ (only one of which is defined).  
%\todo{Maybe we have to define this extension differently, so that the sum converges to the integral.}

We say that a choice of Gaussian distributions on $W_{n,d}$ is \emph{coherent} if there exists $0<\lambda\leq 1$ such that $P_{d, \lambda}$ converges pointwise (as $d$ goes to infinity) to an integrable \emph{nonzero} function $\psi$, and all the sequence is dominated by a function with a subgaussian tail:
\begin{equation}\label{coherence}P_{d, \lambda}\to \psi \quad \textrm{pointwise and dominated by a subgaussian tailed function.}\end{equation}
In other words we assume all the sequence of functions $P_{d,\lambda}(x)$ is bounded by $c_1e^{-c_2 x^2}$ for some positive constants $c_1, c_2>0.$

\begin{remark}The normalization condition \eqref{norm} requires that we take the same scaling exponent $\lambda$
inside and outside of $p_d$; otherwise one could work as well with non-normalized weights, but should consider the family $p_d(d^{\lambda} x)d^\alpha$ and ensure it rescales to a nonzero integrable function for some $\alpha, \lambda$. 
%Since the statistic we are considering ($b_0=$ number of connected components) is invariant under multiplication of the random $f$ by a nonzero constant, we can always make this assumption without affecting the outcome, in favor of a simpler exposition.

\end{remark}
The following Lemma will be useful in the sequel.

\begin{lemma}\label{lemma:rescaling}Under the above rescaling assumption there exist $a>0$ and $b>0$
such that we have the asymptotic (as $d \rightarrow \infty$):
$$
 \sum_{d - \ell \in 2 \mathbb{N}} \ell^{a} ({p_d(\ell)})^b \sim \frac{d^{\lambda(a-b+1))}}{2}\int_{0}^{\infty}x^a\psi(x)^bdx,
$$
and the following upper bound:
$$
 \sum_{d - \ell \in 2 \mathbb{N}} \ell^{a} ({p_d(\ell)})^b \leq \frac{ c_1^b \Gamma\left(\frac{a+1}{2}\right)}{(b c_2 )^{\frac{a+1}{2}}} d^{\lambda(a-b+1)}.
$$
\end{lemma}
\begin{proof}
First note that twice the sum is equal to 
$$\int_0^\infty \{\ell\}^a p_d(\ell)^b  d\ell + O(1) = \int_0^\infty (\ell^a + O(\ell^{a-2})) p_d(\ell)^b  d\ell,$$
where $\{\ell \}$ denotes the nearest integer to $\ell$ with the same parity as $d$.
We make a change of variables $\ell = d^\lambda x$, $d\ell = d^\lambda dx$:

\begin{align}\label{eq:chvar}
  \int_0^\infty (\ell^a + O(\ell^{a-2})) p_d(\ell)^b  d\ell &= \int_0^\infty ( (d^\lambda x)^a + O((d^\lambda x)^{a-2})) p_d(d^\lambda x)^b (d^\lambda) dx \\
 &= d^{\lambda(a-b+1)} \int_0^\infty (x^a + O(x^{a-2} d^{-2\lambda}) ) (d^\lambda)^b p_d(x d^\lambda)^b  dx \\
 &\sim d^{\lambda(a-b+1)} \int_0^\infty x^a \psi(x)^b dx,
\end{align}
by Lesbesgue's dominated convergence theorem.

Moreover, since the sequence of integrands is dominated by $x^a (c_1e^{-c_2x^2})^b$,
\begin{equation*}
 \int_0^\infty \ell^a p_d(\ell)^b  d\ell\leq \int_{0}^{\infty}x^a(c_1e^{-c_2x^2})^bdx= \frac{c_1^b}{2}\frac{\Gamma\left(\frac{a+1}{2}\right)}{(b c_2 )^{\frac{a+1}{2}}}.
\end{equation*}
This proves the second part of the lemma.
\end{proof}

\section{Upper bounds and random matrix theory}

\subsection{Counting maxima and minima of a Gaussian field}In order to prove an upper bound for the average number of components of: 
$$X=\{[x]\in \RP^n\, |\, f(x)=0\},$$ 
we perform the following preliminary reduction. Consider the variable $x=(x_0, \ldots, x_n)$ on the sphere $S^n$ and $\hat{x}=(x_1, \ldots, x_n)$ on the unit disk $D_n=\{\|\hat{x}\|^2\leq 1\}$  in $\RR^{n}.$ In this way, given the random polynomial $p$ we can define the Gaussian fields $\hat{p}_{\pm}$ on the disk $D_n$ by:
\begin{equation}\label{phat}\hat{p}_{\pm}(\hat{x})=p(\pm \sqrt{1-\|\hat{x}\|^2}, x_1, \ldots, x_n).\end{equation}
We can consider as well the ``double cover'' of $X$, namely the set $\overline{X}\subset S^{n}$ defined by:
\begin{equation}\label{cover}\overline{X}=\{x\in S^n\,|\, f(x)=0\} \quad\textrm{satisfying}\quad b_0(\overline{X})=2b_0(X)-\frac{1+(-1)^{d+1}}{2}.\end{equation}
The correction term is due to the fact when the degree is odd there is one component of $X$ which lifts to one single component of $\overline{X}$ (the reader can think at the case of a projective line lifting to a circle).

\begin{prop}\label{prop:criticalbound}
$$\EE b_{0}(\overline{X})\leq 4\EE \# \{\textrm{minima of $\hat p_{+}$}\}$$
\end{prop}
\begin{proof}Assume that $f=0$ is a regular equation (hence $X$ is smooth): this happens with probability one. Let $x$ be a point on the sphere $S^n$ where $f$ and its differential do not vanish and consider the stereographic projection $s: S^n\backslash \{x\} \to \RR^n$. Then $g=f\circ s^{-1}$ is a smooth function on $\RR^n$, the number of components of its zero set and its critical points correspond through $s$ to those of $f$. Each component $C$ of $Y=\{g=0\}$ is compact and separates $\RR^n$ into two disjoint open sets, one of which is bounded: we call it the \emph{interior} of $C$ and denote it by $I(C)$. For every component $C$ of $Y$ consider the set $A(C)$ defined by taking the closure of $I(C)\backslash \{\textrm{interiors of components of $Y$ that are contained in $I(C)$}\}.$ 
This set $A(C)$ is a compact manifold with boundary; since $g$ is zero on this boundary, it has a maximum and a minimum on $\textrm{int}(A(C))$. 
In particular $b_0(\overline{X})=b_0(\{g=0\})$ is bounded by the number of maxima plus the number of minima of $g$, which is the same as the number of maxima plus the number of minima of $f|_{S^n}.$
Every minimum or maximum of $f$ is a minimum or a maximum for $\hat{p}_+$ or $\hat{p}_-$, except for those critical points lying on the equator $\{x_0=0\}$. 
Since the probability that a critical point of $f|_{S^n}$ is on the equator is zero and $\hat{p}_\pm$ have the same distribution, 
the result follows.
\end{proof}

We recall the following Theorem from \cite{Fyodorov}. 

\begin{thm}\label{thm:fyodorov}Let $\hat{p}:D_n\to \RR$ be a random Gaussian 
field with covariance structure 
that can be expressed in terms of a twice-differentiable univariate function $F$ as follows:
$$\EE\{\hat{p}(\hat{x})\hat{p}(\hat{y})\}=F\left(\langle \hat x, \hat y\rangle+\sqrt{1-\|\hat x\|^2}\cdot \sqrt{1-\|\hat{y}\|^2}\right), \quad \hat{x}, \hat{y}\in D_n,$$
and define the number:
$$B=\frac{F''(1)-F'(1)}{F''(1)+F'(1)}.$$
Then the expectation of the number of minima $\mathcal{N}_m$ of $\hat{p}$ is given by:
\begin{equation}\label{minima}\EE \mathcal{N}_m =(1+B)^{\frac{n+1}{2}}(1-B)^{-\frac{n}{2}}\int_{-\infty}^{+\infty}e^{-\frac{(n+1)Bt^2}{2}}\frac{d}{dt}\mathcal{F}_{n+1}(t)dt\end{equation}
where $\mathcal{F}_{n+1}(t)$ is the probability density for the largest eigenvalue of a $\emph{\textrm{GOE}}(n+1)$ matrix.
\end{thm}

\subsection{The case of invariant polynomials}
We recall from the previous section that the covariance structure of $f|_{S^n}$ is given by:
$$\EE\{f(x)f(y)\}=\sum_{d-\ell\in 2\mathbb{N}}{p_d(\ell)}^2 \frac{d(n,\ell)}{|S^{n}|}\tilde{C}_{\ell}^{\frac{n-1}{2}}\left(\langle x, y\rangle\right), \quad x,y\in S^n.$$
Thus, the covariance structure of $\hat{p}_{\pm}$ is given by:
$$\EE\{\hat{p}_{\pm}(\hat{x})\hat{p}_{\pm}(\hat{y})\}=\sum_{d-\ell\in 2\mathbb{N}}{p_d(\ell)}^2 \frac{d(n,\ell)}{|S^{n}|}\tilde{C}_{\ell}^{\frac{n-1}{2}}\left(\langle \hat x, \hat y\rangle+\sqrt{1-\|\hat x\|^2}\cdot \sqrt{1-\|\hat{y}\|^2}\right), \quad \hat{x}, \hat{y}\in D_n.$$
This covariance structure satisfies the hypothesis of Theorem \ref{thm:fyodorov}, with corresponding $F$ given by:
\begin{equation}\label{eq:covariance}F(t)=\sum_{d-\ell\in 2\mathbb{N}}{p_d(\ell)}^2 \frac{d(n,\ell)}{|S^{n}|}\tilde{C}_{\ell}^{\frac{n-1}{2}}(t).\end{equation}
In particular, we see that
\begin{equation}\label{f1}F'(1)=\sum_{d-\ell\in 2\mathbb{N}}{p_d(\ell)}^2 \frac{d(n,\ell)}{|S^{n}|}\frac{d}{dt}\tilde{C}_{\ell}^{\frac{n-1}{2}}(t)|_{t=1},\end{equation}
and
\begin{equation}\label{f2}F''(1)=\sum_{d-\ell\in 2\mathbb{N}}{p_d(\ell)}^2 \frac{d(n,\ell)}{|S^{n}|}\frac{d^2}{dt^2}\tilde{C}_{\ell}^{\frac{n-1}{2}}(t)|_{t=1}.\end{equation}
In order to finally compute $B$ appearing in the statement of Theorem \ref{thm:fyodorov} 
it is enough to compute the derivatives $\frac{d}{dt}\tilde{C}_{\ell}^{\frac{n-1}{2}}(1)$ and $\frac{d^2}{dt^2}\tilde{C}_{\ell}^{\frac{n-1}{2}}(1),$
which are provided in the next lemma.

\begin{lemma}\label{lemmaC}$$\frac{d}{dt}\tilde{C}_{\ell}^{\frac{n-1}{2}}(t)|_{t=1}=\frac{(n+\ell-1)\ell}{n}\quad \textrm{and}\quad \frac{d^2}{dt^2}\tilde{C}_{\ell}^{\frac{n-1}{2}}(t)|_{t=1}=\frac{(n+\ell)(n+\ell-1)\ell(\ell-1)}{n(n+2)}.$$
\end{lemma}
\begin{proof}
We recall the following formula relating the non-normalized Gegenbauer polynomials and their derivatives \cite{gege}:
\begin{equation}\label{eq:gege}
\quad C_{\ell}^m(1)={\binom{\ell+2m-1}{\ell}}\quad  \textrm{and}\quad \frac{d}{dt}C_{\ell}^{m}(t)=2m C_{\ell-1}^{m+1}(t).\end{equation}
Thus, in particular we have:
\begin{align*}
\frac{d}{dt}\tilde{C}_{\ell}^{m}(t)|_{t=1}&=2m {\binom{\ell+2m-1}{\ell}}^{-1}C_{\ell-1}^{m+1}(1) \\
&=2m {\binom{\ell+2m-1}{\ell}}^{-1} {\binom{\ell+2m}{\ell-1}}\\
&=\frac{\ell(\ell+2m)}{2m+1}.
\end{align*}
Similarly for the second derivative:
\begin{align*}\frac{d^2}{dt^2}\tilde{C}_{\ell}^{m}(t)|_{t=1}&=4m(m+1) {\binom{\ell+2m-1}{\ell}}^{-1}C_{\ell-2}^{m+2}(1)\\
&=4m(m+1) {\binom{\ell+2m-1}{\ell}}^{-1} {\binom{\ell+2m+1}{\ell-2}}\\
&=\frac{(\ell+2m)(\ell+2m+1)\ell(\ell-1)}{(2m+3)(2m+1)}.
\end{align*}
Evaluating these two quantities at $m=\frac{n-1}{2}$ gives the result.
\end{proof}
As a corollary we can explicitly write the quantities appearing in Theorem \ref{thm:fyodorov} for our random invariant polynomial. Setting:
\begin{small}$$f^{(1)}(n,\ell)=\frac{(n+\ell-1)\ell}{n}=\frac{\ell^2}{n}+O(\ell)\quad\textrm{and}\quad f^{(2)}(n,l)=\frac{(n+\ell)(n+\ell-1)\ell(\ell-1)}{n(n+2)}=\frac{\ell^4}{n(n+2)}+O(\ell^3)$$\end{small}
we obtain:
\begin{align}\label{eq:B}
B=& \frac{F''(1)-F'(1)}{F''(1)+F'(1)}=\frac{\sum {p_d(\ell)}^2 d(n,\ell) (f^{(2)}(n,\ell)-f^{(1)}(n,\ell))}{\sum {p_d(\ell)}^2  d(n,\ell)(f^{(2)}(n,\ell)+f^{(1)}(n,\ell))}\\
\label{eq:1+B}
1+B=& 2\frac{F''(1)}{F''(1)+F'(1)}=2\frac{\sum {p_d(\ell)}^2 d(n,\ell) f^{(2)}(n,\ell)}{\sum {p_d(\ell)}^2 d(n,\ell)(f^{(2)}(n,\ell)+f^{(1)}(n,\ell))}\\
\label{eq:1-B}
1-B=&2\frac{F'(1)}{F''(1)+F'(1)}=2\frac{\sum {p_d(\ell)}^2 d(n,\ell) f^{(1)}(n,\ell)}{\sum {p_d(\ell)}^2 d(n,\ell) (f^{(2)}(n,\ell)+f^{(1)}(n,\ell))}.
\end{align}
\subsection{Large degree asymptotics for coherent ensembles}
We discuss here the asymptotic behavior of \eqref{eq:B} and \eqref{eq:1-B}  for coherent ensembles.
\begin{prop}\label{B:asympt}
For a coherent family of ensembles we have:
$$\lim_{d\to \infty}B=1\quad \textrm{and}\quad (1-B)^{-1}\sim \frac{d^{2\lambda}}{ 2(n+2)}\frac{\int_{0}^{+\infty}x^{n+3}\psi(x)^2dx}{\int_{0}^{+\infty}x^{n+1}\psi(x)^2dx},$$
where recall that $\psi$ is defined by the rescaling limit (\ref{coherence}).
\end{prop}
\begin{proof}
Let us start by recalling equation \eqref{eq:dimharmonic}:
$$d(n,\ell)=\frac{(n+2\ell-1)(n+\ell-2)!}{\ell! (n-1)!}=\frac{2}{(n-1)!}\ell^{n-1}+O(\ell^{n-2}).$$
Substituting this into \eqref{eq:B} we obtain:
 $$B=\frac{\frac{2}{n!}\sum {p_d(\ell)}^2\ell^{n+3}+\sum {p_d(\ell)}^2 O(\ell^{n+2})}{\frac{2}{n!}\sum {p_d(\ell)}^2\ell^{n+3}+\sum p_d(\ell)^2 O(\ell^{n+2})},$$
and in particular, using Lemma \ref{lemma:rescaling}, we can write:
$$\lim_{d\to \infty}B=\lim_{d\to \infty}\frac{\frac{2}{n!}d^{\lambda n}+O(d^{\lambda(n-1)})}{\frac{2}{n!}d^{\lambda n}+O(d^{\lambda(n-1)})}=1.$$
For $(1-B)^{-1}$ we argue similarly and substituting \eqref{eq:dimharmonic} into the reciprocal of \eqref{eq:1-B} we obtain:
\begin{align*}(1-B)^{-1}&=\frac{1}{2}\frac{\frac{2}{(n-1)!}\sum {p_d(\ell)}^2 \frac{\ell^{n+3}}{n(n+2)}+\sum {p_d(\ell)}^2 O(\ell^{n+2})}{\frac{2}{(n-1)!}\sum {p_d(\ell)}^2 \frac{\ell^{n+1}}{n} }\\
&\sim \frac{1}{2(n+2)}\frac{d^{\lambda n}\int_{0}^{+\infty}x^{n+3}\psi(x)^2dx }{d^{\lambda( n-2)}\int_{0}^{+\infty}x^{n+1}\psi(x)^2dx+O(d^{\lambda (n-3)})}\\
&=\Theta(d^{2\lambda}),
\end{align*}
where in the last line we have used the assumption that $\psi$ is nonzero, hence both the integrals $\int_0^{\infty}x^{n+3}\psi(x)^2dx$ and $\int_0^{\infty}x^{n+1}\psi(x)^2dx$ are different from zero.
\end{proof}
As a corollary we derive the following theorem.

\begin{thm}[The upper bound]\label{thm:upperbound}
For a coherent family of ensembles:
$$\EE b_{0}(X) = O(d^{\lambda n}).$$
\end{thm}
\begin{proof}
We prove the claim for $\overline{X}$ (the ``double cover'' of $X$), and the result for $b_0(X)$ follows from \eqref{cover}.

From Proposition \ref{prop:criticalbound} we know that $\EE b_0(\overline{X})\leq 4\EE \mathcal{N}_m;$ on the other hand combining Proposition \ref{B:asympt} into equation \eqref{minima}, we obtain:
$$\EE \mathcal{N}_m =(1+B)^{\frac{n+1}{2}}(1-B)^{-\frac{n}{2}}\int_{-\infty}^{+\infty}e^{-\frac{(n+1)Bt^2}{2}}\frac{d}{dt}\mathcal{F}_{n+1}(t)dt = \Theta(d^{\lambda n}),$$
since as $d$ goes to infinity $B\to 1$ and consequently by the dominated convergence theorem we have:
$$\int_{-\infty}^{+\infty}e^{-\frac{(n+1)Bt^2}{2}}\frac{d}{dt}\mathcal{F}_{n+1}(t)dt\to \int_{-\infty}^{+\infty}e^{-\frac{(n+1)t^2}{2}}\frac{d}{dt}\mathcal{F}_{n+1}(t)dt=I_{n+1},$$
 which is independent of $d$. 
 \end{proof}

\section{Lower bounds using the barrier method}
In this section, we will prove the following theorem; we generalize the construction given in \cite{LerarioLundberg} for the real Fubini-Study ensemble.

\begin{thm}[The lower bound]\label{thm:lowerbound}
Given a coherent family of ensembles there exists a constant $c>0$ such that:
$$\EE b_{0}(X)\geq c d^{\lambda n}.$$
\end{thm}

As above, we prove the claim for $\overline{X}$ (the ``double cover'' of $X$) and the result for $b_0(X)$ follows from \eqref{cover}.

The proof of Theorem \ref{thm:lowerbound}
can be reduced to a local consideration.
 For $x$ a point on the sphere $S^n$ consider the following event:
\begin{equation}\label{omega}
\Omega(x, r)=\{f(x)>0\quad \textrm{and} \quad f |_{\partial D(x, r)}<0\}.\end{equation}
If $\Omega(x, r)$ occurs, then $\{f=0\}$ has a component inside $D(x, r)$ (and one of these components \emph{loops} around $x$).
We will prove that for $r \sim \rho_n d^{-\lambda}$, with $\rho_n$ a constant specified in (\ref{eq:defr}), there exists a constant $a_1>0$ independent of $x$ and $d$ such that
\begin{equation}\label{eq:P}
\PP\{\Omega(x, r)\}\geq a_1.
\end{equation}
Since we can cover the sphere $S^n$ with at least $k  d^{\lambda n}$ disjoint  such disks, for some $k>0$,
this immediately gives us the statement (each one of these disks contributing at least $a_1$ to the expectation in the statement).

We will need two Lemmas in order to prove (\ref{eq:P}).
The first provides the existence of the so called \emph{barrier} function; 
the second gives a bound on the expected maximum of $|f|$ restricted to $\partial D(x,r)$.

\begin{lemma}\label{lemma1}
Given a coherent family of ensembles, there exists a constant $c_1>0$ such that for every $d>0$ and every $x$ in $S^n$ there exists a homogeneous polynomial $B_x$ of degree $d$ and norm one (with respect to the scalar product in $W_{n,d}$ inducing the probability distribution) satisfying:
$$B_x(x)\geq c_1 d^{\frac{\lambda(n-2)}{2}} \quad \textrm{and}\quad B_x|_{\partial D(x, r)}\leq -c_1 d^{\frac{\lambda(n-2)}{2}}.$$
\end{lemma}

\begin{lemma}\label{lemma2} 
For a coherent family of ensembles, there exists a constant $c_2>0$ such that:
$$\mathbb{E} \max_{\partial D(x, r)}|f|\leq c_2 d^{\lambda(n-2)/2}.$$
\end{lemma}

\subsection{Proof of Lemma \ref{lemma1}}

Choose constants $0< A < B$ such that $\int_{A}^{B} \psi(x) dx > 0$.
We will choose 
\begin{equation}\label{eq:defr}
r=\frac{2y_n}{2Bd^\lambda+n-1},
\end{equation}
where $y_n$ is the first point after zero where the Bessel function $J_{\frac{n-2}{2}}$ reaches a minimum.

We define our barrier function $B_x$ by normalizing the function
$$\sum_{\ell \in [A d^\lambda, B d^\lambda]} {p_d(\ell)} Y_\ell(\theta).$$
Since the terms in this sum are orthonormal, the $L^2$-norm is simply the square root of the number of terms 
Card$[A d^\lambda, B d^\lambda] = \Theta(d^{\lambda/2})$, and we have:
$$B_x := \frac{\sum_{\ell \in [A d^\lambda, B d^\lambda]} {p_d(\ell)} Y_\ell(\theta)}{\sqrt{\text{Card}[A d^\lambda, B d^\lambda]}} .$$

Evaluating at the North pole, we have $Y_\ell(0) = \Theta( \ell^{(n-1)/2})$,
and in the range $A d^\lambda \leq \ell \leq B d^\lambda$, $Y_\ell(0) = \Theta( d^{\lambda(n-1)/2})$, so
$$ B_x(x) = \Theta \left( d^{-\lambda/2} d^{\lambda(n-1)/2} \sum_{\ell \in [A d^\lambda, B d^\lambda]}  {p_d(\ell)} \right) = \Theta \left( d^{\lambda (n-2)/2} \right) .$$

In order to see that $B_x$ is negative with the same order on $\pd D$,
first we write $Y_\ell$ in terms of a Jacobi polynomial:
$$Y_\ell(\theta)=c(n,\ell) P_{\ell}^{(\frac{n-2}{2},\frac{n-2}{2})}(\cos \theta),$$
where $P_{\ell}^{(\frac{n-2}{2},\frac{n-2}{2})}$ is a Jacobi polynomial
and $c(n,\ell)$ has order $\Theta(\ell^{1/2})$ \cite[Appendix]{LerarioLundberg}.

Using Szeg\"o's generalization of the Mehler-Heine asymptotic (\cite[Thm. 8.21.12]{Szego}) we have:
$$P_{\ell}^{(\frac{n-2}{2},\frac{n-2}{2})}(\cos \theta) = \bigg(\sin \frac{\theta}{2}\cdot\cos \frac{\theta}{2}\bigg)^{\frac{2-n}{2}}\bigg\{ h(n,\ell)\bigg(\frac{\theta}{\sin \theta}\bigg)^{1/2}J_{\frac{n-2}{2}}(N \theta)+R_\ell(\theta)\bigg\},$$
where $N=\frac{2\ell + n-1}{2},$ $h(n,\ell)=\Theta(1)$, and the error term $R_\ell(\theta)$ is always less than $\theta^{1/2}O(\ell^{-\frac{3}{2}}).$

Since by definition $r=\frac{2y_n}{2Bd^\lambda+n-1}$ we have $(\sin \frac{r}{2}\cdot\cos \frac{r}{2})^{\frac{2-n}{2}}=\Theta(d^{\lambda \frac{n-2}{2}}),$ and plugging all this into the definition of $B_x$, 
we have:
$$B_x(r) = \Theta(d^{-\frac{\lambda}{2}})\Theta(d^{\lambda \frac{n-2}{2}})\sum_{\ell \in [A d^\lambda, B d^\lambda]} {p_d(\ell)} \Theta(\ell^{1/2})\bigg\{\bigg(\frac{r}{\sin r}\bigg)^{1/2}J_{\frac{n-2}{2}}(r(2\ell+n-1)/2)+R_\ell(r)\bigg\}.$$

Note that the argument of $J_{\frac{n-2}{2}}$ satisfies:
\begin{equation}
 r(2\ell+n-1)/2 = y_n \frac{2 \ell + n - 1}{2Bd^\lambda+n-1} \leq y_n.
\end{equation}
Moreover, by increasing the value $A$ if necessary, we may assume that, for all $\ell \in [Ad^\lambda, B d^\lambda]$, 
the value of $r(2\ell+n-1)/2$ is within an interval $[(1 - \epsilon)y_n,y_n]$ where $J_{\frac{n-2}{2}}$ is decreasing and $J_{\frac{n-2}{2}}((1 - \e)y_n)< 0$.

Then, for $B_x(r)$ we have:
$$B_x(r) = \Theta(d^{\lambda \frac{n-3}{2}})\sum_{\ell\in [Ad^\lambda, B d^\lambda]}{p_d(\ell)} \Theta(\ell^{1/2})J_{\frac{n-2}{2}}((1-\e) y_n) \leq - c_1 d^{\lambda \frac{n-2}{2}} ,$$
where we have neglected the lower-order error terms $R_\ell(r) = r^{1/2} O(\ell^{-3/2})$.
This concludes the proof of Lemma \ref{lemma1}.

\subsection{Proof of Lemma \ref{lemma2}}

Expanding $f$ in an orthonormal basis of spherical harmonics and restricting to $\pd D(x,r)$,
we have (following the notation in \cite{LerarioLundberg}):

$$ f|_{\pd D(x,r)} = f(r,\phi) =  \sum_{m=0}^{d} \sum_{j \in I_m} \hat{\xi}_{m,j} (\sin r)^m Y_j(\phi) = \sum_{m=0}^{d} (\sin r)^m \sum_{j \in I_m} \hat{\xi}_{m,j} Y_j(\phi), $$ 
where
 \begin{equation}\label{eq:newxi}
  \hat{\xi}_{m,j} =\sum_{\ell\in L_m} {p_d(\ell)} \xi_{\ell,m,j} N_\ell^m P^{\left( \frac{n-1}{2}+m \right)}_{\ell-m}(\cos r) ,\quad j\in I_m, \quad \xi_{\ell,m,j} \sim N(0,1).
 \end{equation}

By the addition formula for independent Gaussians we have:
$$\hat{\xi}_{m,j} \sim N(0,\sigma(m,d)^2 ),$$ where:
\begin{equation}\label{eq:sigma}
\sigma(m,d)^2 = \sum_{\ell=m}^d \big\{{p_d(\ell)} N_\ell^m P^{\left( \frac{n-1}{2}+m \right)}_{\ell-m}(\cos r)\big\}^2 .
 \end{equation}

Using the triangle inequality, we bound the expectation of 
$$M:=\| f|_{\pd D(x,r)} \|_\infty = \max_{\phi \in S^{n-1}} |f(r,\phi)|$$ as:
$$\mathbb{E}M\leq \sum_{m=0}^{d} (\sin r)^m \EE \max_{\phi \in S^{n-1}} \left| \sum_{j \in I_m} \hat{\xi}_{m,j} Y_j(\phi) \right| .$$
The function:
$$ F(\phi)=\sum_{j \in I_m} \hat{\xi}_{m,j} Y_j(\phi)$$
is a random spherical harmonic on $S^{n-1}$ with independent identically distributed Gaussian coefficients
(since $\sigma(m,d)$ is independent of $j$ for each fixed $m$).  
This allows us to use the following bound for the expected sup-norm (over the whole sphere $S^{n-1}$) which is
based on a reverse-H\"older inequality, an estimate for spherical harmonics giving an $L^\infty$ bound in terms of the $L^2$-norm \cite[Appendix]{LerarioLundberg}:

\begin{lemma}\label{lemma:sup}
Let $F(\phi)$ be a random spherical harmonic.  Then
$$\EE \max_{\phi \in S^{n-1}} \left|F(\phi) \right| \leq \sigma(m,d) \cdot d(n-1,m),$$ 
where $d(n-1,m)$ is the dimension of the space of spherical harmonics of degree $m$ in $n$ variables.
\end{lemma}

Applying Lemma \ref{lemma:sup} to $\EE M$:
\begin{equation}\label{max}\mathbb{E}M\leq \sum_{m=0}^{d} (\sin r)^m \sigma(m,d) \cdot d(n-1,m).\end{equation}
Next we use an estimate\footnote{This is based on the known maximum of the Gegenbauer polynomials \cite[Formula (7.33.1)]{Szego}.} from \cite{LerarioLundberg}
$$(N_\ell^m)^2 \left( P^{\left( \frac{n-1}{2}+m \right)}_{\ell-m}(\cos r) \right)^2 \leq a_0 (2\ell+n-1) {\binom{\ell + m + n - 2}{\ell-m}} ,$$
in order to bound $\sigma(m,d)$ as follows:
 \begin{align*}\sigma(m,d)^2& \leq  a_0 \sum_{\ell=m}^d (2\ell+n-1) {\binom{\ell + m + n - 2}{\ell-m}} {p_d(\ell)}^2 \\
 &\leq \frac{a_0 (2n)^{2m+n-1}}{(2m + n - 2)!}\sum_{\ell=0}^d \ell^{2m+n-1} {p_d(\ell)}^2 .
 \end{align*}
 where $a_0$ is a constant that does not depend on $d$.
 We use the second part of Lemma \ref{lemma:rescaling}:
  \begin{align*}
 \frac{a_0 (2n)^{2m+n-1}}{(2m + n - 2)!}\sum_{\ell=0}^d \ell^{2m+n-1} {p_d(\ell)}^2 &\leq  \frac{a_0 (2n)^{2m+n-1}}{(2m + n - 2)!} d^{\lambda(2m+n-2)} \frac{2^{-1-m-n/2}\Gamma(m+n/2)}{c_2^{m+n/2}} \\
 &\leq  \frac{(c_3d^\lambda)^{2m}\Gamma\left(m+\frac{n}{2}\right)}{(2m+n-2)!} d^{\lambda(n-2)}\leq \frac{(c_3d^{\lambda})^{2m}}{m!}d^{\lambda(n-2)}.
 \end{align*}
  
In particular $\sigma(m,d) \leq  \frac{(c_3d^{\lambda})^{m}}{\sqrt{m!}}d^{\frac{\lambda(n-2)}{2}}$.
 Recalling that $r=\frac{2y_n}{2Bd^\lambda+n-1}\leq \frac{y}{d^\lambda}$ (we have set $y=y_n / B$ for simplicity of notation) 
 and using $(\sin r)^m \leq r^m \leq \left( \frac{y}{d^\lambda} \right)^m$, we have in (\ref{max}):
 $$\EE M\leq \sum_{m=0}^d \left( \frac{y}{d^\lambda} \right)^m \cdot \sigma(m,d) \cdot d(n-1,m) \leq a_1 d^{\lambda(n-2)/2}\sum_{m=0}^d \frac{c_4^{m}}{\sqrt{m!}} \cdot d(n-1,m).$$
 We estimate $d(n-1,m) \leq a_2 m^{n-2} \leq {a_4}^m$, so that using the fact that $\sum_m x^m/\sqrt{m!}=g(x)$ is a convergent power series, we obtain:
$$\EE M \leq g(c_5)d^{\lambda(n-2)/2} .$$
 This concludes the proof of Lemma \ref{lemma2}.

\subsection{Proof of Theorem \ref{thm:lowerbound}}
As explained in the comments following the statement of Theorem \ref{thm:lowerbound}, it suffices to prove the lower bound (\ref{eq:P}).

Using $\B$ provided by Lemma \ref{lemma1} we can decompose $W_{n,d}=\textrm{span}\{\B\}\oplus\textrm{span}\{\B\}^\perp$, 
thus getting the following decomposition for $f$:
$$f=\xi_0 \B+f^\perp\quad\textrm{with}\quad \xi_0\sim N(0, 1).$$
We let now $\tilde \xi_0$ be a random variable distributed as $\xi_0$ but independent of it and we define ${f}_\pm=\pm \tilde \xi_0 \B+f^\perp$. 
Notice that both $f$ and ${f}_\pm$ have the same distribution. 
The introduction of these new random polynomials allows us to write:
\begin{equation}\label{splitf}f=\xi_0 \B+\frac{1}{2}({f}_{+}+{f}_{-})\end{equation}
and to split our problem into the study of the behavior of $\B$ and ${f}_\pm$ separately. 
In fact the event $\Omega(x, r)$ happens provided that for some constant $a_2>0$ the two following events both happen:
\begin{itemize}
\item[1)]$E(x,r)=\left\{\xi_0\B(x)\geq 2a_2 d^{\frac{\lambda(n-2)}{2}} \textrm{ and } \xi_0B_x|_{\partial D(x, r)}\leq -2a_2 d^{\frac{\lambda(n-2)}{2}}\right\}$;
\item[2)]$G(x,  r)=\left\{ |{f}_\pm(x)| \leq a_2 d^{\frac{\lambda(n-2)}{2}}\textrm{ and }\|{f}_{\pm}|_{\partial D(x, r)}\|_\infty \leq a_2d^{\frac{\lambda(n-2)}{2}} \right\}.$
\end{itemize}
To check that $E(x,r)\cap G(x,r)$ implies $\Omega(x, r)$ we simply substitute the inequalities defining $E(x,r)$ and $G(x, r)$ in (\ref{splitf}) evaluating respectively at $x$ and at $\partial D(x, r)$.

Now by definition the two events $E(x,r)$ and $G(x, r)$ are independent and thus:
$$\PP\{E(x,r)\cap G(x, r)\}=\PP\{E(x,r)\}\PP\{G(x, r)\}.$$
Because of this it suffices to show that the probability of each one of them is bounded from below by a positive constant that does not depend on $d$ and this will provide a bound from below for the probability of $\Omega(x, r)$ which is independent of $d$.

The fact that the probability of $E(x,r)$ is bounded from below independently on $d$ immediately follows from Lemma \ref{lemma1}: the random variable $\xi_0\B(x)d^{-\frac{\lambda(n-2)}{2}}$ is Gaussian with mean zero and variance $\Theta(1)$; similarly for the boundary. Thus for $a_3$ small enough the probability of it being bigger than $a_3$ is uniformly bounded from below.

It remains to study $G(x, r)$; to this extent notice that this event is given by the intersection of the four events 
$G_{1,2}=\{|{f}_\pm(x)| \leq a_2d^{\frac{\lambda(n-2)}{2}}\}$ and  $G_{3,4}=\{\|{f}_\pm|_{\partial D(x, r)}\|_\infty \leq a_2d^{\frac{\lambda(n-2)}{2}}\}$. For these we do not need independence as:
$$\PP\big\{\bigcap_{i} G_i\big\}=1-\PP\big\{\big(\bigcap_{i} G_i\big)^c\big\}=1-\PP\big\{\bigcup_{i} G_i^c\big\}\geq 1-\sum_i \PP\{G_i^c\}$$
and it is enough to prove that the probability of the complement of each one of them is small.

By rotation invariance, we have $\PP \{ G_{1,2}^c \} \leq \PP \{ G_{3,4}^c \}$,
and since $f_\pm$ are each distributed as $f$, 
the probability of each $G_{3,4}^c$ is exactly the probability of the following event:
\begin{equation}\label{eq:event}
 \{\|f|_{\partial D(x, r)}\|_\infty \geq a_2d^{\frac{\lambda(n-2)}{2}} \quad \textrm{for a random} f\in W_{n,d}\}.
\end{equation}

Hence we consider the positive random variable: $$M:=\max_{\partial D(x, r)}|f|.$$
By Lemma \ref{lemma2}, we have $\EE M \leq c_2 d^{\lambda (n-2)/2}$.
Applying Markov's inequality to $M$:
$$\PP  \{M \geq a_2 d^{\lambda(n-2)/2} \} \leq   \frac{c_2}{a_2} .$$

Choosing $a_2$ sufficiently large, we have that the probability of the event (\ref{eq:event}) is at most $c_2/a_2 < 1/5$,
and therefore $\PP \{ G(x,r) \} > 1-4/5 = 1/5$ which concludes the proof.

\section{Slice sampling}
\subsection{The restriction of a random polynomial to a circle}
In this section we restrict our random function to a one-dimensional slice.
We study the average number of zeros of the restriction and compare the asymptotic behavior to above results on connected components.
Notice that restricting to a circle or to a projective line are essentially equivalent approaches since the number of points in $\{f|_{S^1}=0\}$ is twice the number of points of $\{f=0\}$ on $\RP^1$.

The restriction of $f$ to a circle $S^1$ is a classical random univariate polynomial, and we can relate our results to classical studies.

\begin{lemma}\label{lemma:comparisonparameter}
Referring to (\ref{covariance1}), 
the parameters $\beta_0, \ldots, \beta_{\lfloor d/2\rfloor}$ for $f$ and $f|_{S^1}$ coincide. Moreover the number $B$
appearing in Theorem \ref{thm:fyodorov}  also coincides for $f$ and for $f|_{S^1}$.

\end{lemma}
\begin{proof}

We start by noticing that using the description of the covariance structure for $f$ as in \eqref{covariance1} we obtain the following for the covariance structure of $f|_{\{x_2=\cdots= x_n=0\}}$:
$$F(x_0,x_1,y_0,y_1)=\sum_{k=0}^{\lfloor \frac{d}{2}\rfloor}\beta_k\|(x_0,x_1)\|^{2k}\|(y_0,y_1)\|^{2k}\langle (x_0,x_1), (y_0, y_1)\rangle^{d-2k}$$
and we see that the parameters $\beta_0, \ldots, \beta_{\lfloor d/2\rfloor}$ for $f$ and $f|_{S^1}$ coincide.

On the other hand writing the covariance structure $F$ of $\hat{p}_{\pm}$ (defined by \eqref{phat}) using the expression \eqref{covariance1} we obtain
$$F=\sum \beta_k F_k\quad \textrm{where}\quad 
F_k(t)=t^{d-2k}.$$
In this way we immediately see that $F_k'(1)=d-2k$ and $F_k''(1)= (d-2k)(d-2k-1)$. Substituting this into the definition of $B$, we obtain:
\begin{equation}\label{eq:Bbeta}B=\frac{\sum_{k}\beta_k(d-2k)^2-2\sum_k \beta_k(d-2k)}{\sum_k \beta_k (d-2k)^2}=1-2\frac{\sum_k \beta_k(d-2k)}{\sum_k \beta_k (d-2k)^2}.\end{equation}
In particular the number $B$ for $f$ is the same as for $f|_{S^1}$.
\end{proof}

\subsection{The parameter of a random polynomial}
Given a random polynomial in $W_{n,d}$ whose covariance structure is given by \eqref{covariance1} we define the number $\delta(f)$ (the \emph{parameter of the distribution}):
$$\delta(f)=\sqrt{\frac{\sum_{k=0}^{\lfloor d/2\rfloor} \beta_k(d-2k)}{\sum_{k=0}^{\lfloor d/2\rfloor}\beta_k}}$$
(alternatively if $F$ is given in the form \eqref{covariance2} the parameter can be computed from \cite[Thm 5.7]{Kostlan}).
The parameter gives the average number of real zeros of the restriction of $f$ to any $\RP^1$ (or one half the number of zeros of the restriction of $f$ to any fixed circle $S^1\subset S^n$).

\begin{lemma}\label{lemma:paras}For a coherent family of ensembles:
$$\delta(f)=\Theta(d^{\lambda})$$

\end{lemma}
\begin{proof}
We need to use the expression for $\delta(f)$ in terms of the weights $r_i$ defined in \cite{Kostlan}:
$$\delta(f)=\sqrt{\frac{\sum_{i=0}^{\lfloor d/2\rfloor} (d-2i)(d-2i+n-1)r_i^2}{n\sum_{i=0}^{\lfloor d/2\rfloor}r_i^2 }}.$$
Using equation \eqref{eq:ptor} we can rewrite this as:
$$\delta(f)=\sqrt{\frac{\sum \ell(\ell+n-1)d(n,\ell){p_d(\ell)}^2}{n\sum d(n,\ell){p_d(\ell)}^2}}.$$
Using now  $\ell(\ell+n-1)d(n,l)=\Theta(\ell^{n+1})$ and $d(n,\ell)=\Theta(\ell^{n-1})$ together with Lemma \ref{lemma:rescaling} gives the stated asymptotic.

\end{proof}
Thus we can immediately derive the following corollary.

\begin{cor}[Slice sampling]\label{coro:sampling}For a coherent family of ensembles:
$$\EE b_0(X)=\Theta \left(\EE b_0(X\cap \RP^1)^n\right).$$
\end{cor}
Motivated by \eqref{eq:Bbeta} we define also:
\begin{equation}\label{delta'}\delta'(f)=\sqrt{\frac{\sum_{k=0}^{\lfloor d/2\rfloor}  \beta_k (d-2k)^2}{\sum_{k=0}^{\lfloor d/2\rfloor}  \beta_k (d-2k)}}
\end{equation}
and in this way $(1-B)^{-1}=2\delta'(f)^2$, as it follows immediately from \eqref{eq:Bbeta}.

We prove now Proposition \ref{prop:minima}.
\begin{proof}
Notice that Lemma \ref{lemma:comparisonparameter} implies:
$$\delta'(f)=\delta'(f|_{S^1})$$
(and $\delta(f)=\delta(f|_{S^1})$ as well). Moreover for a coherent family of ensembles we have:
\begin{align*}\EE \#\{\textrm{minima of $f$}\}&=2(1+B)^{\frac{n+1}{2}}(1-B)^{-\frac{n}{2}}\int_{-\infty}^{+\infty}e^{-\frac{(n+1)Bt^2}{2}}\frac{d}{dt}\mathcal{F}_{n+1}(t)dt\\
& \sim 2^{n+\frac{3}{2}} \left(\delta(f)' \right)^nI_{n+1}\\
&= 2^{n+\frac{3}{2}} \left(\delta(f|_{S^1})' \right)^nI_{n+1}\\
&\sim c'_n\left(\EE \#\{\textrm{minima of $f|_{S^1}$}\}\right)^n.
\end{align*}
\end{proof}

\section{Comparison with \cite{Sodin}}\label{sectionNS}

In this section we discuss briefly the relations between the above methods and those of Nazarov and Sodin from \cite{Sodin}.

In order to match notations, let us recall the definitions from \cite{Sodin}. The authors consider a sequence of Gaussian ensembles $H_d$ of functions on a Riemannian manifold which in our case is the sphere $S^n$ or the projective space $\RP^n$; the gaussian ensemble is defined using a scalar product on $H_d$. With respect to this scalar product, the authors consider the reproducing kernel $K_d$, defined by:
$$f(y)=\langle f(\cdot), K_d(\cdot, y)\rangle_{H_d}.$$
In our case $H_d=\RR[x_0, \ldots, x_n]_d$ and the scalar product is orthogonally invariant. This reproducing kernel is exactly the covariance structure of the random $f$ (which the authors normalize to $K_d(x,x)=1$). 

In particular, using their notation, we see that:
$$K_d(x, y)=\frac{\sum_{d-l\in 2\mathbb{N}}p_d(\ell)^2d(n, \ell)Q_\ell^n(\cos \theta(x,y))}{\sum_{d-l\in 2\mathbb{N}}p_d(\ell)^2d(n,l)}$$
where $\theta(x,y)$ is the angle between $x$ and $y$, the function $Q_\ell^n$ is what we denoted by $\tilde{C}^{\frac{n-1}{2}}_{\ell}$, and the numerator comes from the normalization assumption, using $Q_\ell^n(\cos \theta(x,x))=1$.

Denoting by $\phi:T_xS^n\simeq \RR^n\to S^n$ the exponential map, the standing assumption from \cite{Sodin} is that:
\begin{equation}\label{NS}\lim_{d \to \infty} \sup_{\|u\|, \|v\|\leq R}|K_d(\phi(d^{-1}u), \phi(d^{-1}v))-k(u-v)|=0\end{equation}
where in fact the dependence on the point $x\in S^n$ is irrelevant, being the distribution rotational invariant.
As discussed in \cite[Example 2.5.4]{Sodin}, it is indeed possible to choose different scaling for $d$ (not just the linear one), 
and one can consider for example the condition:
\begin{equation}\label{NSl}\lim_{d\to \infty} \sup_{\|u\|, \|v\|\leq R}|K_d(\phi(d^{-\lambda}u), \phi(d^{-\lambda}v))-k(u-v)|=0\quad \textrm{for some $0<\lambda\leq 1$}.\end{equation}
Using this notation Theorem 4 from \cite{Sodin} can be read as follows. Assuming the \emph{nondegeneracy} of the gaussian ensembles \cite[Definition 3]{Sodin} and that the \emph{limiting spectral measure} has no atoms \cite[Section 1.1]{Sodin}, if condition \eqref{NSl} holds, then there exists a (possibly zero) constant $\overline{\nu}$ such that:
\begin{equation}\label{NST}\lim_{d\to \infty}\mathbb{E}\left\{\left|\frac{b_0(X)}{d^{\lambda n}} -\overline{\nu}\textrm{Vol}(S^n)\right|\right\}=0.\end{equation}
The rescaling condition \eqref{NSl} is implied by our coherence assumption,
and the rescaled covariance structure turns out to be the Fourier transform of $\psi(r)$ viewed as a function
of the radial coordinate $r$ in $\RR^n$.
Let us explain this in the simple case $n=1$
(the cases $n>1$ are similar but require asymptotics of the Gegenbauer polynomials).
We have
$$K_d(\theta)= \sum_{d-\ell\in 2\mathbb{N}} p_d(\ell)^2 \cos(\ell \theta) \sim \int_0^\infty \psi(x )^2 \cos(x d^\lambda \theta) dx.$$
Thus,
$$K_d(d^{-\lambda} \theta) \sim \int_0^\infty \psi(x)^2 \cos(x \theta) dx.$$
The subgaussian tail of $\psi$ implies that this Fourier transform is continuous (even smooth),
so that condition \eqref{NSl} is satisfied.
This implies that $\mathbb{E}b_0(X)/d^{\lambda n}$ indeed has a limit as $d\to \infty$.

\section{Proof of Theorem \ref{estimates}}
As in the above paragraphs, we consider the double cover $\overline{X}\subset S^n$ of $X\subset \RP^n$, the two are related by \eqref{cover}. Let us start by recalling that:
$$\lim_{d\to \infty}\frac{\EE b_0(\overline{X})}{d^{\lambda n}}\leq \lim_{d\to \infty}4\frac{\EE \mathcal{N}_m}{d^{\lambda n}}$$
where $\EE \mathcal{N}_m$ is the expectation of the number of minima of $f$ (the defining polynomial of $X$) on one \emph{half-sphere} (see Theorem \ref{thm:fyodorov}),
and the existence of the limit of $\EE b_0(\overline{X})/d^{\lambda n}$ follows from 
the results in \cite{Sodin} discussed in the previous section.
Recall also, from Proposition \ref{B:asympt}, that:
\begin{equation}\label{estmoments}\lim_{d\to \infty}B=1\quad \textrm{and}\quad (1-B)^{-1} \sim \frac{d^{2\lambda}}{2(n+2)}\frac{\int_{0}^\infty x^{n+3}\psi(x)^2 dx}{\int_{0}^\infty x^{n+1}\psi(x)^2 dx},
\end{equation}
where $B = \frac{F''(1)-F'(1)}{F''(1)+F'(1)}$.
Thus, using the above notation $\mu_{n}$ to denote the moments of $\psi^2$, we have:
$$(1-B)^{-1} \sim \frac{d^{2\lambda}}{2(n+2)}\frac{\mu_{n+3}}{\mu_{n+1}}.$$
Using the identity \eqref{minima}, we can thus write:
\begin{align}\label{extrema}\lim_{d\to \infty}\frac{\EE b_0(\overline{X})}{d^{\lambda n}} &\leq \lim_{d\to \infty}\frac{4}{d^{\lambda n}}(1+B)^{\frac{n+1}{2}}(1-B)^{-\frac{n}{2}}\int_{-\infty}^{+\infty}e^{-\frac{(n+1)Bt^2}{2}}\frac{d}{dt}\mathcal{F}_{n+1}(t)dt\\
&= 4 \cdot 2^{(n+1)/2}\left(\frac{1}{2(n+2)}\frac{\mu_{n+3}}{\mu_{n+1}} \right)^{n/2}\int_{-\infty}^{+\infty}e^{-\frac{(n+1)t^2}{2}}\frac{d}{dt}\mathcal{F}_{n+1}(t)dt\\
&\label{eqfinal}=2^{\frac{5}{2}}\left(\frac{1}{n+2}\frac{\mu_{n+3}}{\mu_{n+1}}\right)^{n/2}I_{n+1}.
\end{align}

The integral $I_{n+1}$ can be evaluated asymptotically when $n\to \infty$ \cite[Equation 82]{Fyodorov}:
\begin{equation}\label{asI}I_{n+1}\sim g_{n+1}\doteq A\cdot 2^{\frac{35}{16}}(n+1)^{-\frac{17}{36}}e^{-(n+1)+\frac{4 \sqrt{2}}{3}\sqrt{n}}\quad \textrm{where}\quad \log A= -\frac{169}{96}\log 2+\frac{1}{2}\zeta'(1),\end{equation}
and $\zeta'(1)$ is the derivative at one of the Riemann Zeta function ($\zeta'(1)\approx -0.1654$).

\begin{remark}For the case of \emph{curves} ($n=2$), the integral $I_3$ can be evaluated exactly and yields:
\begin{equation}\label{I3}I_3=\frac{27}{2\sqrt{2}\pi}\int_{\RR}\int_{-\infty}^t\int_{-\infty}^te^{-3t^2-\frac{3}{2}\lambda_1^2-\frac{3}{2}\lambda_2^2}(t-\lambda_1)(t-\lambda_2)|\lambda_1-\lambda_2|d\lambda_1d\lambda_2dt=\frac{1}{\sqrt{6}}\end{equation}
\end{remark}
%The asymptotic in \eqref{asI} means that for every $\epsilon>0$ there exists $n_{\epsilon}$ such that:
%$$\left |\frac{I_n}{g_n}-1\right|\leq \epsilon \quad \forall n>n_\epsilon, \textrm{ which implies}\quad I_n\leq (1+\epsilon)g_n \quad  \forall n>n_\epsilon.$$
%Consider now a constant $C'_1$ such that $I_n\leq C_1 g_n$ for all $n\leq n_\epsilon$, then we obtain:
%$$I_n\leq C_1 g_n \quad \forall n\in \mathbb{N},\quad \textrm{whit $C_1=\max\{(1+\epsilon), C_1'\}$}.$$
%We remark that this is the only part of the proof where appears a nonconstructive constant; in any case this constant still does not depend on the choice of the coherent distribution.

%In particular we obtain:
%\begin{align*}\limsup_{d \to \infty}\frac{\EE b_0(\overline{X})}{d^{\lambda n}} &\leq C_1\cdot  (e^{-1} A 2^{75/16})\left(\frac{e}{2}\right)^{-n}n^{-\frac{17}{36}}e^{\frac{4\sqrt{2}}{3}\sqrt{n}}\left(\frac{\mu_{n+3}}{\mu_{n+1}}\right)^{n/2}\\
%&=\overline{c}_1n^{-c_2}e^{c_3 \sqrt{n}}\left(\frac{e}{2}\right)^{-n}\left(\frac{\mu_{n+3}}{\mu_{n+1}}\right)^{n/2}.
%\end{align*}
%The statement for $X$ now follows from \eqref{cover} after setting $c_1=\overline{c}_1/2.$

\section{Examples}\label{sec:Examples}

Some of these examples have already been studied in \cite{GaWe3, LerarioLundberg, NazarovSodin, Sarnak, Sodin}.

\begin{example}[Kostlan ensemble]\label{Kostlanex}
Let's start by considering the case $n=1$. In this case the zero locus of $f$ consists of points on $S^1$; the classical setting of univariate polynomials over the real line is obtained by defining $\hat{f}(t)=f(1,t)$ (the \emph{dehomogenization} of $f$): the average number of zeros of a random $\hat{f}$ on $\RR$ equals $\frac{1}{2}\EE b_0(f)$. If the covariance structure of $f$ is given by:
$$\EE_{\textrm{Kostlan}}\{f(x)f(y)\}=\langle x,y\rangle^d, \quad x,y\in \RR^2,$$
$f$ is said to be a \emph{Kostlan} random polynomial (notice that the covariance structure of the corresponding random field $\hat{f}:\RR\to \RR$ is given by $\EE\{\hat{f}(t)\hat{f}(s)\}=(1+ts)^d$).  Concerning $\EE b_0(f)$, when $n=1$, one has the \emph{exact} result \cite{EdelmanKostlan95}:
$$\EE_{\textrm{Kostlan}} b_0(f)=2\sqrt{d}=\Theta(d^{1/2})\quad (n=1).$$

For the general $n>1$ the covariance structure of $f\in W_{n,d}$ is given by:
$$\EE_{\textrm{Kostlan}}\{f(x)f(y)\}=\langle x, y \rangle^d,\quad x,y\in\RR^{n+1}.$$

This ensemble is coherent, and the rescaling exponent is $\lambda=1/2$.
In order to see this, 
first we use Equation (\ref{eq:ptor}) and \cite[Eq. (8)]{Kostlan}.

\begin{equation}
 {p_d(\ell)}^2 = \frac{|S^n|}{G_{n,d} d(n,\ell)} \frac{\Gamma(\frac{n+1}{2}) \left(\frac{n-1}{2} + \ell \right) \Gamma \left( n-1 + \ell \right) d!}{2^{d-1} (n-1)! \ell! \left(\frac{d - \ell}{2} \right)! \Gamma \left( \frac{n+1}{2} + \frac{d+\ell}{2} \right)},
\end{equation}
where $G_{n,d}$ denotes the normalization constant needed to satisfy (\ref{norm}).

Next using Equation (\ref{eq:dimharmonic}) for $d(n,\ell)$ and simplifying,
\begin{equation}
 {p_d(\ell)}^2 = \frac{|S^n|}{G_{n,d}} \Gamma \left( \frac{n+1}{2} \right) \frac{ d!}{2^{d} \left(\frac{d - \ell}{2} \right)! \Gamma \left( \frac{n+1}{2} + \frac{d+\ell}{2} \right)}.
\end{equation}

Multiplying and dividing by $\left( \frac{d+\ell}{2} \right)! = \Gamma \left( \frac{d+\ell}{2} + 1 \right)$, we have
\begin{equation}
 {p_d(\ell)}^2 = \frac{|S^n|}{G_{n,d}} \Gamma \left( \frac{n+1}{2} \right)  \frac{ \Gamma \left( \frac{d+\ell}{2} + 1 \right) }{ \Gamma \left( \frac{n+1}{2} + \frac{d+\ell}{2} \right)} \frac{ 1}{2^{d} } \binom{d}{\frac{d-\ell}{2}}.
\end{equation}
As $d \rightarrow \infty$ we have \cite{Olver}:
$$  \frac{ \Gamma \left( \frac{d+\ell}{2} + 1 \right) }{ \Gamma \left( \frac{n+1}{2} + \frac{d+\ell}{2} \right)} = \left( \frac{d+\ell}{2} \right)^{-(n-1)/2}(1+ O(1/d)). $$

\begin{equation}
p_d(\ell)^2 = \frac{|S^n| \Gamma \left( \frac{n+1}{2} \right)}{G_{n,d} \sqrt{2}(d/2)^{(n-2)/2} } (1+ O(1/d)) \left( 1 + \frac{\ell}{d} \right)^{-(n-1)/2}  \frac{ 1}{2^{d-1} \sqrt{d} } \binom{d}{\frac{d-\ell}{2}}.
\end{equation}
As in the classical central limit of the binomial distribution, by Stirling's formula,
under the rescaling $\ell = x d^{1/2}$, we have
$$ \frac{ 1 }{2^{d-1} \sqrt{d}} \binom{d}{\frac{d-x \sqrt{d}}{2}} \sim \frac{e^{-x^2/2}}{d\sqrt{2 \pi}}.$$
Since $\left( 1 + \frac{x}{\sqrt{d}} \right)^{-(n-1)/2} \sim 1,$ we conclude that
$\frac{|S^n| \Gamma \left( \frac{n+1}{2} \right)}{G_{n,d} \sqrt{2}(d/2)^{(n-2)/2}}$ converges to a constant $c$ and that
$$ \sqrt{d} p_d(x\sqrt{d}) \sim c \frac{e^{-x^2/4}}{(2 \pi)^{1/4} } .$$

In general, for the Kostlan distribution Gayet and Welschinger \cite{GaWe2, GaWe3} have recently proved that (see also \cite[Section 2.5.4]{Sodin}):
$$\EE_{\textrm{Kostlan}} b_0(f)=\Theta (d^{n/2}).$$
This result can be immediately recovered by applying Corollary \ref{coro:sampling}, 
since the Kostlan ensemble is coherent (with rescaling exponent $\lambda=1/2$).

We proceed now with the estimation of the leading coefficient of $d^{n/2}$ in $\EE b_0(X).$ First notice that:
\begin{equation}\label{ko}\mu_{k}^{\textrm{Kostlan}}=\int_{0}^{\infty}c^2\frac{e^{-x^2/2}}{\sqrt{2\pi}}x^kdx=\frac{c^2 2^{k/2}}{2 \sqrt{\pi}}\Gamma\left(\frac{k+1}{2}\right)\end{equation}
which provides for this case:
$$\frac{1}{n+2}\frac{\mu_{n+3}}{\mu_{n+1}}=\frac{2\Gamma\left(\frac{n}{2}+2\right)}{(n+2)\Gamma\left(\frac{n}{2}+1\right)}=\frac{\Gamma\left(\frac{n}{2}+2\right)}{\left(\frac{n}{2}+1\right)\Gamma\left(\frac{n}{2}+1\right)}=1.$$
Plugging \eqref{ko} into \eqref{eqfinal} we obtain exponential decay of the leading coefficient\footnote{In fact for the Kostlan case a stronger decay was shown in \cite{GaWe2} involving a factor $e^{-cn^2}$.}:
\begin{equation}\label{kofinal}\lim_{d\to \infty}\frac{\EE b_0(X)}{d^{n/2}}\leq 2^{\frac{3}{2}} I_{n+1}\leq c_1 n^{-c_2}e^{-n+c_3\sqrt{n}}.\end{equation}
For the case of curves $(n=2)$, we can use \eqref{I3} and obtain:
$$\lim_{d\to \infty}\frac{\EE b_0(X)}{d}\leq \frac{2}{\sqrt{3}}\approx 1.1547$$
%\todo{update this together with $\frac{\mu_{n+3}}{\mu_{n+1}}=\frac{2\Gamma(n/2+2)}{\Gamma(n/2+1)}\sim n$}

\end{example}

\begin{example}[real Fubini-Study (RFS) ensemble]\label{RFSex}
In terms of the weights ${p_d(\ell)}$, the real Fubini-Study ensemble is simpler: the ${p_d(\ell)}$ are constant.
Normalizing so that the sum is unity, we have ${p_d(\ell)} \sim 2/d $.
The rescaling exponent is $\lambda=1$, and $d p(x d)$ rescales to the characteristic function of $[0,1]$.

The real Fubini-Study ensemble (the name is due to P. Sarnak, see \cite{Sarnak}) is obtained 
 by sampling at random a polynomial uniformly from the unit sphere in the $L^2_{S^n}$-norm in $W_{n,d}$.
In the case $n=1$ (using techniques from \cite{EdelmanKostlan95}) one has again the exact result \cite{LerarioLundberg}:
$$\EE_{\textrm{RFS}} b_0(f)=2\sqrt{\frac{d(d+1)}{3}} \quad (n=1).$$
For the number of extrema, using $I_2=\sqrt{3}/(2\sqrt{2})$, one obtains (see \cite{Nicolaescu2} for more details):
 $$\EE\#\{\textrm{extrema of $f|_{S^1}$}\}\sim2d\sqrt{\frac{3}{5}}$$

In the general case, it is proved in \cite{LerarioLundberg} that:
$$\EE_{\textrm{RFS}}b_0(f)=\Theta(d^n).$$
In view of Milnor's (deterministic) bound $b_0(f)= O(d^n)$, the interesting part of this statement is the lower bound, which is obtained using a technique similar to the one introduced in \cite{NazarovSodin} for the study of nodal domains of random spherical harmonics on $S^2.$

Again this result is recovered using Corollary \ref{coro:sampling} from the knowledge of the case $n=1$.
As above we write:
$$\EE b_0(X)=a_nd^n+o(d^{n})$$

Since in this case $\psi=\chi_{[0,1]}$ all its moments are easily computable, and since $\mu_k((\chi_{[0,1]})^2)=(k+1)^{-1}$ we have:
$$\frac{\mu_{n+3}}{(n+2)\mu_{n+1}}=\frac{1}{n+4}.$$
Plugging this into Theorem \ref{estimates}, one obtains exponential decaying of the coefficients $a_n$:
$$a_n \leq C_1n^{-\frac{n}{2}-c_2}e^{-n+c_3 \sqrt{n}}\quad \textrm{for some constants $C_1, c_2, c_3>0$}.$$
In particular, the average number of minima of a RFS polynomial on $S^2$ of degree $d$ is:
$$\EE \#\{\textrm{minima of $f|_{S^2}$}\}\sim \frac{2^{3/2}}{6}I_3d^2=\frac{d^2}{3\sqrt{3}}$$
and consequently the Nazarov-Sodin constant for $b_0(X)$ (in the projective plane) satisfies:
$$a_2=\lim_{d\to\infty}\frac{\EE b_0(X)}{d^2}\leq  \frac{1}{3\sqrt{3}}\approx 0.1924$$
In fact M. Nastasescu \cite{Nastasescu} has studied this constant numerically and obtained the approximation $a_2\approx 0.0195$.
\end{example}

\begin{example}[Ensembles with a prescribed order of growth]Let $\psi:\RR_+\to \RR$ an integrable function with a subgaussian tail (for example take $\psi=\chi_{[0,1]}$) and $0<\lambda\leq1$. Define:
$$p_d(\ell)=\frac{1}{d^{\lambda}}\psi\left(\frac{\ell}{ d^{\lambda}}\right).$$
Then by construction we have $p_d(xd^{\lambda})d^{\lambda}=\psi(x)$, and the above theorem applies yielding for this model:
$$\EE b_0(X)=\Theta(d^{\lambda n}).$$
More generally, one can consider a sequence $\{g(d)\}_{d\in \mathbb{N}}$ with $0<g(d)\leq d$ and set $p_d(\ell)=g(d)^{-1/n}\psi(g(d)^{-1/n}\ell)$. The proofs of the above theorems still work in this case, providing $\EE b_0(X)=\Theta(g(d))$.
\end{example}

\begin{example}[Random spherical harmonics]We discuss in this example the special case when we give all the ``power'' to the top coefficient, i.e.
$$p_d(\ell)\equiv 0\quad\textrm{except for}\quad p_{d}(d)=1.$$ 
As a model for random polynomials, 
this case is degenerate; it produces random spherical harmonics of degree $d$ (recall that the representation of $O(n+1)$ in the space of spherical harmonics is irreducible).
Interest in this model comes from studies related to "quantum chaos" \cite{NazarovSodin}.

Notice that our rescaling assumptions are not satisfied in this case, but we can still use Theorem \ref{thm:fyodorov} to estimate the number of extrema of a random spherical harmonic $f$ of degree $d$  on the sphere $S^n$

In this case we can follow the line of the proof of Theorem \ref{estimates} (the first line of \eqref{extrema} does not require the rescaling assumption and holds in general for random fields with orthogonal invariance). 

We start by writing equations \eqref{f1} and \eqref{f2} for the covariance structure of random spherical harmonics; using Lemma \ref{lemmaC} we obtain:
$$F'(1)=|S^n|\frac{(n+d-1)!(n+2d-1)}{(d-1)!n!}\quad\textrm{and}\quad F''(1)=|S^n|\frac{(n+d)!(n+2d-1)}{(d-2)!n!(n+2)}.$$
Plugging these formulas into \eqref{eq:B} and \eqref{eq:1-B} we obtain:
$$ \lim_{d\to \infty}B=1\quad\textrm{and}\quad (1-B)^{-1}\sim \frac{d^2}{2(n+2)}.$$
Thus, using equation \eqref{extrema} we have:
\begin{align}\lim_{d\to\infty}\frac{\EE\#\{\textrm{extrema of $f|_{S^n}$}\}}{d^n}&= \lim_{d\to \infty}\frac{4}{d^{\lambda n}}(1+B)^{\frac{n+1}{2}}(1-B)^{-\frac{n}{2}}\int_{-\infty}^{+\infty}e^{-\frac{(n+1)Bt^2}{2}}\frac{d}{dt}\mathcal{F}_{n+1}(t)dt\\
&=2^{\frac{5}{2}}(n+2)^{-\frac{n}{2}}I_{n+1}\\
&\sim C_1 n^{-c_2 -\frac{n}{2}}e^{-n-c_3\sqrt{n}}.\end{align}
For example, using the line before the last one and $I_2=\sqrt{3}/(2\sqrt{2})$ and $I_3=1/\sqrt{6}$, we obtain:
$$\EE\#\{\textrm{extrema of $f|_{S^1}$}\}\sim2d\quad \textrm{and} \quad \EE\#\{\textrm{extrema of $f|_{S^2}$}\}\sim\frac{d^2}{\sqrt{3}}$$

Again, by Proposition \ref{prop:criticalbound}, we can estimate $\EE b_0(X)$ using (one-half) this bound; M. Nastasescu studied the leading coefficient for the case $n=2$ numerically in this case as well.
Our estimate exceeds the value she found, $ 0.0598 $, by an approximate factor of $4.8$.
For the related setting of random plane waves, M. Krishnapur has obtained an upper bound (by counting the number of vertical tangents of the nodal set)
that exceeds the experimental prediction by a factor of approximately $3.6$ \cite{Krishnapur}.

\end{example}

%\begin{example}
 %\todo{Prescribed order of growth $d^\lambda$.}
%\end{example}

%\begin{example}
 %\todo{Is there any interesting basis that leads to $\lambda$ not equal to $1$ or $1/2$?}
%\end{example}


\begin{thebibliography}{9}

\bibitem{Agrachev} A. A. Agrachev: \emph{Topology of quadratic maps and Hessians of smooth maps},  Itogi nauki. VINITI. Algebra. Topologiya. Geometriya 26
      (1988), 85-124.

\bibitem{AgrachevLerario} A. A. Agrachev, A. Lerario: \emph{Systems of quadratic inequalities}, Proceedings of the London Mathematical Society 105 (2012), 622-660.


\bibitem{ABAC} A. Auffinger, G. Ben Arous, J. Cerny: \emph{Random matrices and complexity of spin glasses}, Comm. Pure. Appl. Math. 66 (2013), Issue 2, 165-201

\bibitem{ABA} A. Auffinger, G. Ben Arous: \emph{Complexity of random smooth functions on the high-dimensional sphere}, Ann. Probab. 41 (2013) Issue 6, 4214-4247

\bibitem{Bihan} F. Bihan: \emph{Asymptotiques de nombres de Betti d'hypersurfaces projectives r�elles}, 	arXiv:math/0312259.

\bibitem{Burgisser}P. B\"urgisser: \emph{Average Euler characteristic of random algebraic varieties}, C. R. Acad. Sci. Paris, 345 (2007), 507-512.


\bibitem{EdelmanKostlan95}
A. Edelman, E. Kostlan: \emph{How many zeros of a random polynomial are real?}, Bull. Amer. Math. Soc. 32 (1995), 1-37.


\bibitem{Fyodorov} Y. V. Fyodorov: \emph{High-Dimensional Random Fields and Random Matrix Theory}, arXiv:1307.2379,
to appear in the special issue of Markov Processes and Related Fields (2015) devoted to Leonid Pastur.

 \bibitem{Fyodorov2} Y. V. Fyodorov: \emph{Complexity of random energy landscapes, glass transition, and absolute value of the spectral determinant of random matrices}, Phys. Rev. Lett. 92 (2004), No.24,  240601
 
\bibitem{Fyodorov3} Y. V. Fyodorov, C. Nadal: \emph{Critical Behavior of the Number of Minima of a Random Landscape at the Glass Transition Point and the Tracy-Widom Distribution}, Phys. Rev. Lett. 109 (2012), No. 16,  167203
 
\bibitem{GayetWelschinger2011}
D. Gayet, J-Y. Welschinger: \emph{Exponential rarefaction of real curves with many components}, 
Publications math\'ematiques de l'IH\'ES, 113 (2011), 69-96.


%\bibitem{GaWe1}
%D. Gayet, J-Y. Welschinger: \emph{What is the total Betti number of a random real hypersurface?}, to appear in Journal de Crelle.

\bibitem{GaWe2}
D. Gayet, J-Y. Welschinger: \emph{Betti numbers of random real hypersurfaces and determinants of random symmetric matrices}, arXiv:1107.2288v1. 

\bibitem{GaWe3}
D. Gayet, J-Y. Welschinger: \emph{Lower estimates for the expected Betti numbers of random real hypersurfaces}, to appear in J. London Math. Soc.

\bibitem{GaWe4}
D. Gayet, J-Y. Welschinger: \emph{Topology of random real algebraic submanifolds},  arXiv:1307.5287.

\bibitem{Kac} 
M. Kac: \emph{On the average number of real roots of a random algebraic equation},  Bull. Amer. Math. Soc. Volume 49, Number 4 (1943), 314-320.

\bibitem{KhOr}V. Kharlamov, S. Yu. Orevkov: \emph{Asymptotic growth of the number of classes of real plane algebraic curves when the degree increases}
Journal of Mathematical Sciences 113 (2003) 666-674
 
\bibitem{Kostlan} E. Kostlan: \emph{On the expected number of real roots of a system of random polynomial equations}, Foundation of computational mathematics, Proceedings of the Smalefest.

\bibitem{Krishnapur}
M. Krishnapur, (private communication)

\bibitem{Lerario2012} A. Lerario: \emph{Random matrices and the expected topology of quadric hypersurfaces},  to appear in Proc. Amer. Math. Soc. 

\bibitem{LerarioLundberg} A. Lerario, E. Lundberg: \emph{Statistics on Hilbert's Sixteenth Problem}, to appear in IMRN.

\bibitem{LeLu2} A. Lerario, E. Lundberg: \emph{Gap probabilities and applications to geometry and random topology}, arXiv:1309.5661

\bibitem{gege}Dae San Kim, Taekyun Kim, and Seog-Hoon Rim: \emph{Some identities involving Gegenbauer polynomials}, Advances in Difference Equations 2012, 2012:219

\bibitem{Nastasescu} 
M. Nastasescu: \emph{The number of ovals of a real plane curve}, Senior Thesis, Princeton 2011.
Thesis and Mathematica code available at: {\verb http://www.its.caltech.edu/mnastase/Senior_Thesis.html }


\bibitem{NazarovSodin} F. Nazarov,  M. Sodin: \emph{On the Number of Nodal Domains of Random Spherical Harmonics}, 	 American Journal of Mathematics
131(2009), 1337-1357
 
 \bibitem{Nicolaescu1} L. Nicolaescu: \emph{ Complexity of random smooth functions on compact manifolds}, to appear Indiana U. Math. J
 \bibitem{Nicolaescu2} L. Nicolaescu: \emph{Fluctuations of the number of critical points of random trigonometric polynomials}, preprint

\bibitem{Olver}
F. W. J. Olver: \emph{On an asymptotic expansion of a ratio of gamma functions}, Proc. Roy. Irish Acad. Sect. A 95 (1995), 5-9.


\bibitem{Sarnak}
P. Sarnak: \emph{Letter to B. Gross and J. Harris on ovals of random plane curves} (2011) available at: {\verb http://publications.ias.edu/sarnak/section/515 }

\bibitem{Sarnak2} P. Sarnak, I. Wigman: \emph{Topologies of nodal sets of random band limited functions}
arXiv:1312.7858

%\bibitem{Sodin} M. Sodin: \emph{Online slides for the Summer School on Quantum Chaos
%ESI, Vienna 2012}, available at: http://ipht.cea.fr/Pisp/stephane.nonnenmacher/ESI2012/Sodin-slides.pdf

\bibitem{Sodin} M. Sodin: \emph{Lectures on random nodal portraits}, preprint. Lecture notes for a mini-course given at
the St. Petersburg Summer School in Probability and Statistical Physics (June, 2012), available
at: http://www.math.tau.ac.il/sodin/SPB-Lecture-Notes.pdf

\bibitem{SW}
E. M. Stein, G. Weiss: \emph{Introduction to Fourier analysis on Euclidean spaces}, Princeton Mathematical Series, No. 32, Princeton University Press, Princeton, N.J., 1971, x+297.

\bibitem{Szego} 
G. Szeg\"o: \emph{Orthogonal Polynomials}, American Math. Soc. Colloquium Publications, Vol XXIII, 1959.

\bibitem{Wilson}
G. Wilson: \emph{Hilbert's sixteenth problem}, Topology 17 (1978), 53-74.


\end{thebibliography}
\end{document}